\documentclass[11pt, english, twoside]{article}
\usepackage[margin=3cm]{geometry}
\usepackage{amsbsy,amsmath,amsthm,amssymb}

\usepackage{fourier}
\DeclareMathAlphabet{\mathcal}{OMS}{cmsy}{m}{n}

\usepackage[dvipsnames]{xcolor}
\usepackage{bm}
\usepackage{esint}
\usepackage{babel}
\usepackage[colorlinks=true, citecolor = blue]{hyperref}

\usepackage{indentfirst}
\usepackage{changepage, cases}	
\usepackage{setspace}
\usepackage{indentfirst}
\usepackage{graphicx}
\usepackage{calc}
\usepackage{orcidlink}
\usepackage{booktabs,multirow}
\usepackage{makecell}

\usepackage{fancyhdr}
\allowdisplaybreaks

\theoremstyle{plain}
\newtheorem{theorem}{Theorem}[section]

\newtheorem{corollary}[theorem]{Corollary}

\theoremstyle{remark}
\newtheorem{remark}[theorem]{Remark}

\theoremstyle{definition}

\newenvironment{proof of theorem 1.1}{{\noindent \em Proof of Theorem 1.1.}}{\hfill $\Box$\par}
\newenvironment{proof of theorem 1.2}{{\noindent \em Proof of Theorem 1.2.}}{\hfill $\Box$\par}

\DeclareSymbolFont{EulerExtension}{U}{euex}{m}{n}
\DeclareMathSymbol{\euintop}{\mathop} {EulerExtension}{"52}
\DeclareMathSymbol{\euointop}{\mathop} {EulerExtension}{"48}

\begin{document}
	\title{Erd\'{e}lyi-type integrals for $F_K$ function and their $q$-analogues}
	\author{
		Liang-Jia Guo \orcidlink{0000-0002-1146-5896} $^{\rm 1}$,
		Min-Jie Luo \orcidlink{0000-0001-7433-4490} $^{\rm 2}$\thanks{Corresponding author \\ E-mail address: \texttt{mathwinnie@live.com} (Min-Jie Luo).}}
	\date{}
	\maketitle
	\begin{center}\small 
		$^{1}$\emph{School of Mathematics and Statistics, \\ Donghua University, Shanghai 201620, \\ 
			People's Republic of China.}\\
		E-mail: \texttt{guoliangjia77@outlook.com}
	\end{center}
	\begin{center}\small
		$^{2}$\emph{School of Mathematics and Statistics, \\ Donghua University, Shanghai 201620, \\ 
			People's Republic of China.}\\
		E-mail: \texttt{mathwinnie@live.com}, \texttt{mathwinnie@dhu.edu.cn}
	\end{center}
	
	
	\begin{abstract}
		
		In this paper, we revisit the recent result of Luo, Xu, and Raina [Fractal Fract. 6 (3) (2022)] on an Erd\'{e}lyi-type integral for Saran's three-variable hypergeometric function $F_K$. 
		We provide a new proof of this integral and derive an attractive new integral related to Appell's function $F_2$. A further extension on the $L$-variable $F_K$ function,  which appears in physics, is also discussed. Furthermore, we prove various $q$-Erd\'{e}lyi-type integrals for the $q$-analogue of the $F_K$-function. An interesting discrete analogue is also included. We also provide a valuable compilation of the sources for known Erd\'{e}lyi-type integrals of many different hypergeometric functions in the Appendix. \\
		
		\\
		
	\end{abstract}	
	
\tableofcontents 	

\section{Introduction}\label{Introduction}

\subsection{Background and motivation}

The Gauss hypergeometric function ${}_{2}F_{1}$ is defined by 
\[
{}_{2}F_{1}\left[\begin{matrix}
	\alpha,\beta\\
	\gamma
\end{matrix};z\right]:=\sum_{n=0}^{\infty}\frac{(\alpha)_n(\beta)_n}{(\gamma)_n}\frac{z^n}{n!}~~~(|z|<1),
\]
where $\alpha,\beta\in\mathbb{C}$ and $\gamma\notin\mathbb{Z}_{\leq0}:=\{0,-1,-2,\cdots\}$ (see \cite[p. 56]{Higher Transcendental Function Vol 1} and \cite[p. 384]{NIST Handbook}). It was shown by Euler that \cite[p. 59, Eq. (10)]{Higher Transcendental Function Vol 1}
\begin{numcases}{{}_{2}F_{1}\left[\begin{matrix}
			\alpha,\beta\\
			\gamma
		\end{matrix};z\right]=}
	\int_{0}^{1} (1-zt)^{-\alpha}\mathrm{d}\mu_{\beta,\gamma-\beta}(t)
	~~~(\Re(\gamma)>\Re(\beta)>0),  \label{EulerIntegral-1}\\
	\int_{0}^{1} (1-zt)^{-\beta}\mathrm{d}\mu_{\alpha,\gamma-\alpha}(t)
	~~~(\Re(\gamma)>\Re(\alpha)>0). \label{EulerIntegral-2}
\end{numcases}
Here and throughout the paper, we let  $\mu_{\alpha,\beta}$ be the \emph{Dirichlet measure} defined by \cite[p. 64, Definition 4.4-1]{Carlson-Book}
\begin{equation}\label{DirichletMeasure}
	\mathrm{d}\mu_{\alpha,\beta}(t)
	:=\mathfrak{m}_{\alpha,\beta}(t)\mathrm{d}t,
\end{equation}
where
\[
\mathfrak{m}_{\alpha,\beta}(t):=
\frac{\Gamma(\alpha+\beta)}{\Gamma(\alpha)\Gamma(\beta)}t^{\alpha-1}\left(1-t\right)^{\beta-1}~~~(\min\{\Re(\alpha),\Re(\beta)\}>0).
\]
Obviously, $\displaystyle\int_{0}^{1}\mathrm{d}\mu_{\alpha,\beta}(t)=1$. In general, integrals \eqref{EulerIntegral-1} and \eqref{EulerIntegral-2} \emph{cannot} be directly transformed into each other, but Riemann \cite{Riemann-1857} indicated that both could be derived from a multiple integral. According to the existing literature, Schellenberg \cite{Schellenberg-1892}, Wirtinger \cite{Wirtinger-1902}, Poole \cite{Poole-1938} and Erd\'{e}lyi \cite{Erdelyi-1937} studied this problem. We are particularly interested in Erd\'{e}lyi's and Poole's methods. 

Erd\'{e}lyi's method (see \cite[p. 272--273]{Erdelyi-1937}) is to substitute
\[
(1-zt)^{-\alpha}=\int_{0}^{1}(1-zts)^{-\gamma}\mathrm{d}\mu_{\alpha,\gamma-\alpha}(s)~~~(\Re(\gamma)>\Re(\alpha)>0)
\]
into the first integral \eqref{EulerIntegral-1} to obtain the double integral
\begin{equation}\label{Euler-MellinIntegral-1}
\int_{0}^{1}\int_{0}^{1}
(1-zts)^{-\gamma}\mathrm{d}\mu_{\beta,\gamma-\beta}(t)
\mathrm{d}\mu_{\alpha,\gamma-\alpha}(s), 
\end{equation}
which is symmetric in $\alpha$ and $\beta$. If we integrate with respect to $t$ first, then we obtain the second integral \eqref{EulerIntegral-2}. Note that after introducing the new variables $x_1=s/(1-s)$ and $x_2=t/(1-t)$ and making the substitution $z\rightarrow 1-z$ in Erd\'{e}lyi's formula \eqref{Euler-MellinIntegral-1}, we obtain the double integral
\begin{equation}\label{Euler-MellinIntegral-2}
\int_0^{\infty}\int_{0}^{\infty}\frac{x_1^{\alpha} x_2^{\beta}}{(1+x_1+x_2+zx_1x_2)^\gamma}\frac{\mathrm{d}x_1\mathrm{d}x_2}{x_1x_2}=
\frac{\Gamma(\beta)\Gamma(\gamma-\beta)\Gamma(\alpha)\Gamma(\gamma-\alpha)}{(\Gamma(\gamma))^2}
{}_{2}F_{1}\left[\begin{matrix}
	\alpha,\beta\\
	\gamma
\end{matrix};1-z\right].
\end{equation}
The integral on the left-hand side of \eqref{Euler-MellinIntegral-2} is also known as the \emph{Euler-Mellin integral} (see \cite{Berkesch-Forsgard-Passare-2014} and \cite{Tellander-2024}).

Poole's method \cite{Poole-1938} is to apply the integration by parts to the contour integral representation of ${}_{2}F_{1}$, which is very delightful. After reading Poole's paper, Erd\'{e}lyi \cite{Erdelyi-1939a} realized that Poole's method can be generalized by using the fractional derivatives. By inventing the method of \emph{fractional integration by parts}, Erd\'{e}lyi not only gave a simple description of Poole's derivation but also reproduced the Bateman integral
\begin{equation}\label{BatemanIntegral}
{}_{2}F_{1}\left[\begin{matrix}
	\alpha,\beta\\
	\gamma
\end{matrix};z\right]
=\int_{0}^{1}
{}_{2}F_{1}\left[\begin{matrix}
	\alpha, \beta\\
	\lambda
\end{matrix}; zx\right]\mathrm{d}\mu_{\lambda,\gamma-\lambda}(x) \hspace{0.5cm}(\Re(\gamma)>\Re(\lambda)>0).
\end{equation}
In \cite{Erdelyi-1939b}, Erd\'{e}lyi further developed his method and proved the following important integrals for  ${}_{2}F_{1}$: 
\begin{align}
	{}_{2}F_{1}\left[\begin{matrix}
		\alpha,\beta\\
		\gamma
	\end{matrix};z\right]
	&=\int_{0}^{1}\left(1-zx\right)^{-\alpha'}
	{}_{2}F_{1}\left[\begin{matrix}
		\alpha-\alpha', \beta\\
		\lambda
	\end{matrix}; zx\right]\notag\\
	&\hspace{0.5cm}\cdot {}_{2}F_{1}\left[\begin{matrix}
		\alpha', \beta-\lambda\\
		\gamma-\lambda
	\end{matrix}; \frac{(1-x)z}{1-xz}
	\right]\mathrm{d}\mu_{\lambda,\gamma-\lambda}(x) \hspace{0.5cm}(\Re(\gamma)>\Re(\lambda)>0),
	\label{ErdelyiIntegral-1}\\
	{}_{2}F_{1}\left[\begin{matrix}
		\alpha,\beta\\
		\gamma
	\end{matrix};z\right]
	&=\int_{0}^{1}\left(1-zx\right)^{\lambda-\alpha-\beta}
	{}_{2}F_{1}\left[\begin{matrix}
		\lambda-\alpha, \lambda-\beta\\
		\eta
	\end{matrix}; zx\right]\notag\\
	&\hspace{0.5cm}\cdot {}_{2}F_{1}\left[\begin{matrix}
		\alpha+\beta-\lambda, \lambda-\eta\\
		\gamma-\eta 
	\end{matrix}; \frac{(1-x)z}{1-xz}
	\right]\mathrm{d}\mu_{\eta,\gamma-\eta}(x) 
	\hspace{0.5cm}(\Re(\gamma)>\Re(\eta)>0),
	\label{ErdelyiIntegral-2}\\
	{}_{2}F_{1}\left[\begin{matrix}
		\alpha,\beta\\
		\gamma
	\end{matrix};z\right]
	&=\int_{0}^{1}{}_{3}F_{2}\left[\begin{matrix}
		\alpha,\beta,\eta\\
		\lambda,\nu 
	\end{matrix}; zx
	\right]\mathrm{d}\mu_{\eta-\lambda,\gamma-\lambda,\gamma-\lambda+\eta-\nu,\nu}(x)\notag\\ &\hspace{1cm}(\min\{\Re(\nu),\Re(\lambda),\Re(\gamma-\lambda+\eta-\nu)\}>0).
	\label{ErdelyiIntegral-3}
\end{align}
In \eqref{ErdelyiIntegral-3}, $\mu_{\alpha,\beta,\gamma,\eta}$ is the \emph{hypergeometric measure} defined by \cite[p. 6, Eq. (11)]{Luo-Xu-Raina-2022}
\begin{equation}\label{HypergeometricMeasure}
	\mathrm{d}\mu_{\alpha,\beta,\gamma,\eta}(t)
	:=\mathfrak{m}_{\alpha,\beta,\gamma,\eta}(t)
	\mathrm{d}t,
\end{equation}
where $\min\{\Re(\eta),\Re(\gamma),\Re(\eta+\gamma-\alpha-\beta)\}>0$ and
\begin{align*}
\mathfrak{m}_{\alpha,\beta,\gamma,\eta}(t)&:=\frac{\Gamma(\eta+\gamma-\alpha)\Gamma(\eta+\gamma-\beta)}{\Gamma(\eta)\Gamma(\gamma)\Gamma(\eta+\gamma-\alpha-\beta)}
t^{\eta-1}(1-t)^{\gamma-1}
{}_{2}F_{1}\left[\begin{matrix}
	\alpha,\beta\\
	\gamma
\end{matrix}; 1-t
\right].
\end{align*}
We have 
$\displaystyle \int_{0}^{1}\mathrm{d}\mu_{\alpha,\beta,\gamma,\eta}(t)=1$ and
$\mathfrak{m}_{0,\beta,\gamma,\eta}(t)
=\mathfrak{m}_{\alpha,0,\gamma,\eta}(t)
=\mathfrak{m}_{\eta,\gamma}(t)$.

It is particularly noteworthy that Erd\'{e}lyi's integrals are not only of computational interest but also have many important and profound applications. Gasper \cite{Gasper-1975} has shown that \eqref{ErdelyiIntegral-2} provides a natural way to the formulas of Dirichlet-Mehler type for the Jacobi polynomials and the generalized Legendre functions. Virchenko and Fedotova also frequently used the integrals \eqref{ErdelyiIntegral-1} and \eqref{ErdelyiIntegral-2} in their study of the generalized Legendre functions (see \cite{Virchenko-1984}, \cite{Virchenko-Fedotova-1990} and \cite{Virchenko-Fedotova-Book-2001}). Sprinkhuizen-Kuyper \cite{Sprinkhuizen-Kuyper-1979} used \eqref{ErdelyiIntegral-2} to prove a compositional property for her fractional integral operator $I_{\nu}^{\mu,\lambda}$. The integral \eqref{ErdelyiIntegral-2} also plays an important role in the proof of product formula for biangle polynomials \cite{Koornwinder-Schwartz-1997}. The proofs of the compositional properties of Saigo's fractional integral operators $I^{\alpha,\beta,\eta}$ and $J^{\alpha,\beta,\eta}$ are based on integrals \eqref{ErdelyiIntegral-1} and \eqref{ErdelyiIntegral-2} \cite{Saigo-1978}. Raina \cite{Raina-2010} showed that the integral \eqref{ErdelyiIntegral-1} can be used to construct a solution to a certain Abel-type integral equation involving the Appell hypergeometric function $F_3$ in the kernel.  

We refer to integrals that have a form similar to integrals \eqref{ErdelyiIntegral-1}--\eqref{ErdelyiIntegral-3} as \emph{Erd\'{e}lyi-type integrals}. More precisely, for a given hypergeometric function, its Erd\'{e}lyi-type integrals can generally express it as an integral of the product of two functions of the same type (see \cite[p. 2]{Luo-Xu-Raina-2022}). 

Many Erd\'{e}lyi-type integrals for various kinds of univariate and multivariate hypergeometric functions have been discovered using the technique of Erd\'{e}lyi's fractional integration by parts. For example, Luo and Raina \cite{Luo-Raina-2017} recenlty extended \eqref{ErdelyiIntegral-2} to a special class of generalized hypergeometric functions by using the theory of Miller and Paris and then found their applications in the theory of certain generalized fractional integral operator \cite{Luo-Raina-2022}. The table in Appendix \ref{Appendix-1} lists all hypergeometric functions for which Erd\'{e}lyi-type integrals and some of their $q$- and discrete analogues have been established.

\subsection{The plan of the work}

Very recently, there has been a sustained growth in research interest regarding the three-variable hypergeometric function $F_K$ introduced by Saran \cite{Saran-1955a} and many interesting results have been achieved (see \cite{Antonova-Dmytryshyn-Goran-2023}, \cite{Dmytryshyn-Goran-2024}, \cite{Hang-Luo-2025}, \cite{Luo-Raina-2021} and \cite{Luo-Xu-Raina-2022}). The $F_K$-function is defined by
\begin{align}
	&F_K [\alpha_1, \alpha_2, \alpha_2, \beta_1, \beta_2, \beta_1; \gamma_1, \gamma_2, \gamma_3; x, y, z]\notag\\
	&\hspace{0.5cm}:=\sum_{m,n,p=0}^{\infty} \frac{(\alpha_1)_m (\alpha_2)_{n+p} (\beta_1)_{m+p} (\beta_2)_{n}}{(\gamma_1)_m (\gamma_2)_n (\gamma_3)_p m! n! p!}x^m y^n z^p\label{Def-FK}\\
	&\hspace{0.5cm}=\sum_{p=0}^{\infty}
	\frac{(\alpha_2)_p (\beta_1)_p}{p!(\gamma_3)_p}
	{}_{2}F_{1}\left[\begin{matrix}
		\beta_1+p,\alpha_1\\
		\gamma_1 
	\end{matrix};x\right]
	{}_{2}F_{1}\left[\begin{matrix}
	\alpha_2+p,\beta_2\\
	\gamma_2 
	\end{matrix};y\right]z^p,\label{Def-FK-2F1}
\end{align}
where $(x,y,z) \in \mathbb{D}_K := \{(x,y,z) \in \mathbb{C}^3: |x|<1, |y|<1, |z|<(1-|x|)(1-|y|)\}$. 

In 2022, Luo \emph{et al.} \cite{Luo-Xu-Raina-2022} discovered an elegant Erd\'{e}lyi-type integral for Saran's $F_K$-function.
\begin{theorem}[{\cite{Luo-Xu-Raina-2022}}]\label{Th-ErdelyiIntegral-FK}
	Let $\Re(\alpha_1+\eta_1)>\Re(\lambda_1)>0$, $\Re(\beta_2+\mu_2)>\Re(\lambda_2)>0$ and $\Re(\gamma_3)>\Re(\beta_1)>0$. Then we have
	\begin{align}\label{Th-ErdelyiIntegral-FK-1}
		& F_K[\alpha_1, \alpha_2, \alpha_2, \beta_1, \beta_2, \beta_1; \alpha_1+\eta_1, \beta_2+\mu_2, \gamma_3; x,y,z] = \int_0^1 \int_0^1 \int_0^1 (1-ux)^{-\lambda_3} (1-vy)^{-\eta_2}\notag\\
		&\hspace{0.5cm}\cdot F_K[ \alpha_1, \alpha_2-\eta_2, \alpha_2-\eta_2, \beta_1-\lambda_3, \beta_2, \beta_1-\lambda_3; \alpha_1-\lambda_1+\eta_1, \beta_2-\lambda_2+\mu_2, \beta_1-\lambda_3; ux, vy, wz]\notag\\
		&\hspace{0.5cm}\cdot F_K \left[\lambda_1-\eta_1, \eta_2, \eta_2, \lambda_3, \lambda_2-\mu_2, \lambda_3; \lambda_1, \lambda_2, \lambda_3; \frac{(1-u)x}{1-ux}, \frac{(1-v)y}{1-vy}, \frac{wz}{(1-ux)(1-vy)} \right]\notag\\
		&\hspace{0.5cm}\cdot \mathrm{d} \mu_{\alpha_1-\lambda_1+\eta_1, \lambda_1} (u) \mathrm{d} \mu_{\beta_2-\lambda_2+\mu_2, \lambda_2} (v) \mathrm{d} \mu_{\beta_1, \gamma_3-\beta_1}(w),
	\end{align}
	where $(x,y,z) \in \mathbb{D}_K$ and $\mathrm{d}\mu_{\alpha,\beta}(t)$ is defined in \eqref{DirichletMeasure}. 
\end{theorem}

However, the initial proof of integral \eqref{Th-ErdelyiIntegral-FK-1} was relatively difficult. In fact, Luo \emph{et al}. first established a very general integral involving the Srivastava-Daoust function through fractional integration by parts, and then reduced it to \eqref{Th-ErdelyiIntegral-FK-1} by specializing the parameters. In this paper, we shall provide a direct and simple proof of Theorem \ref{Th-ErdelyiIntegral-FK} in Section \ref{Section-3}. Then in Section \ref{Section-4}, we shall derive a curious integral representation for Appell function $F_2$ from Theorem \ref{Th-ErdelyiIntegral-FK}. A further extension, which is related to Belitsky's generalization of $F_K$ \cite[p. 69]{Belitsky-2018}, is considered in Section \ref{Section-3-3}.

Section \ref{Section-5} devotes to the various $q$-Erd\'{e}lyi-type integrals for $q$-analogue of $F_K$. Throughout this paper, we assume that $0<q<1$. For $a\in\mathbb{C}$, we define the $q$-shifted factorials as follows
\[
	(a;q)_0=1, ~~~
	(a;q)_n=\prod_{j=0}^{n-1}(1-aq^j)~~~ (n\in \mathbb{Z}_{\geq1})
\]
and
\[
(a;q)_{\infty}=\lim_{n\rightarrow+\infty}(a;q)_n=\prod_{j=0}^{\infty}(1-aq^j).
\]
The multiple $q$-shifted factorials are defined by
$(a_1,\cdots,a_m;q)_n=(a_1;q)_n\cdots(a_m;q)_n$, where $n$ could be an integer or $+\infty$. 
Let $f:\mathbb{C}^k\rightarrow\mathbb{C}$ be a function. The $k$-dimensional $q$-integral can be defined by (see, for example, \cite[p. 1]{Kaneko-2022} and \cite[p. 241]{Warnaar-2005})
\begin{equation}\label{Def-q-integral}
	\int_{[0,1]^k}f(\mathbf{x})\mathrm{d}_q \mathbf{x}=(1-q)^k\sum_{\mathbf{n}\in\mathbb{Z}_{\geq0}^k}f(q^{\mathbf{n}})q^{\langle\mathbf{n}\rangle}. 
\end{equation}
where $\mathbf{x}=(x_1,\cdots,x_k)$, $\displaystyle
\sum_{\mathbf{n}\in\mathbb{Z}_{\geq0}^k}\equiv \sum_{n_1=0}^{\infty}\cdots \sum_{n_k=0}^{\infty}$ and $\langle \mathbf{n}\rangle=n_1+\cdots+n_k$. 

The $q$-gamma function $\Gamma_q(x)$ is defined by \cite[p. 20, Eq. (1.10.1)]{Gasper-Rahman-2004}
\[
\Gamma_q(x)=\frac{(q;q)_{\infty}}{(q^x;q)_{\infty}}(1-q)^{1-x}
\]
and the $q$-beta function is defined by \cite[p. 22, Eq. (1.10.13)]{Gasper-Rahman-2004}
\[
B_q(x,y)=\frac{\Gamma_q(x)\Gamma_q(y)}{\Gamma_q(x+y)}=(1-q)\frac{(q,q^{x+y};q)_{\infty}}{(q^x,q^y;q)_{\infty}}.
\]
Note that $\displaystyle \lim_{q\rightarrow 1^{-}}\Gamma_q(x)=\Gamma(x)$ and $\displaystyle \lim_{q\rightarrow1^{-}}B_q(x,y)=B(x,y)$. The $q$-integral representation for $B_q(x,y)$ is given by \cite[p. 23, Eq. (1.11.7)]{Gasper-Rahman-2004}
\[
B_q(x,y)=\int_{0}^{1} t^{x-1} \frac{(tq;q)_{\infty}}{(tq^{y};q)_{\infty}}\mathrm{d}_q t, ~~~ \Re(x)>0, ~ y\notin \mathbb{Z}_{\leq 0}. 
\]

The $q$-analogue of $F_K$ can be defined by (see \cite{Ernst-2020} and \cite{Sahai-Verma-2014})
\begin{align}
	&\Phi_K[\alpha_1, \alpha_2, \alpha_2, \beta_1, \beta_2, \beta_1; \gamma_1, \gamma_2, \gamma_3; q, x,y,z]\notag \\
	&\hspace{0.5cm}:=\sum_{m,n,p=0}^{\infty} \frac{(\alpha_1;q)_{m} (\alpha_2;q)_{n+p} (\beta_1;q)_{m+p} (\beta_2;q)_{n}}{(\gamma_1,q;q)_{m} (\gamma_2,q;q)_{n} (\gamma_3,q;q)_{p}}x^m y^n z^p\label{Def-qFK}\\
	&\hspace{0.5cm}=\sum_{p=0}^{\infty}
	\frac{(\alpha_2,\beta_1;q)_p}{(\gamma_3,q;q)_p}
	{}_{2}\phi_{1}\left[\begin{matrix}
		\beta_1 q^p,\alpha_1\\
		\gamma_1 
	\end{matrix};q,x\right]
	{}_{2}\phi_{1}\left[\begin{matrix}
		\alpha_2 q^p,\beta_2\\
		\gamma_2 
	\end{matrix};q,y\right]z^p,\label{Def-qFK-2varphi1}
\end{align}
where $|x|<1$, $|y|<1$, $|z|<1$ and ${}_{r}\phi_{s}$ is the basic hypergeometric series defined by \cite[p. 4, Eq. (1.2.22)]{Gasper-Rahman-2004}
\[
{}_{r}\phi_{s}\left[\begin{matrix}
	a_1,\cdots,a_r\\
	b_1,\cdots,b_s
\end{matrix};q,z\right]:=\sum_{n=0}^{\infty}\frac{(a_1,\cdots,a_r;q)_n}{(b_1,\cdots,b_s,q;q)_n}\Big((-1)^nq^{\binom{n}{2}}\Big)^{1+s-r}z^n.
\]   
For simplicity, we write 
\begin{align}
	&\widetilde{\Phi}_K
	[\alpha_1, \alpha_2, \alpha_2, \beta_1, \beta_2, \beta_1; \gamma_1, \gamma_2, \gamma_3; q, x,y,z]\notag\\
	&\hspace{0.5cm}=\Phi_K\big[q^{\alpha_1}, q^{\alpha_2}, q^{\alpha_2}, q^{\beta_1}, q^{\beta_2}, q^{\beta_1}; q^{\gamma_1}, q^{\gamma_2}, q^{\gamma_3}; q, x,y,z\big]
\end{align}
and
\[
{}_{r}\widetilde{\phi}_{s}\left[\begin{matrix}
	a_1,\cdots,a_r\\
	b_1,\cdots,b_s
\end{matrix};q,z\right]
={}_{r}\phi_{s}\left[\begin{matrix}
	q^{a_1},\cdots,q^{a_r}\\
	q^{b_1},\cdots,q^{b_s}
\end{matrix};q,z\right].
\]

The \emph{only} $q$-integral representation (of Bateman type) for $\Phi_K$  is given by Ernst \cite[p. 10, Theorem 5]{Ernst-2020}:
\begin{align}
	&\widetilde{\Phi}_K[\alpha_1, \alpha_2, \alpha_2, \beta_1, \beta_2, \beta_1; \gamma_1, \gamma_2, \gamma_3; q, x,y,z]\notag\\
	&\hspace{0.5cm}=\int_{[0,1]^3}
	\widetilde{\Phi}_K[\alpha_1, \alpha_2, \alpha_2, \beta_1, \beta_2, \beta_1; \nu_1, \nu_2, \nu_3; q, xu,yv,zw]\notag\\
	&\hspace{1.5cm}\cdot\mathrm{d}\mu_{\nu_1,\gamma_1-\nu_1}(u;q)
	\mathrm{d}\mu_{\nu_2,\gamma_2-\nu_2}(v;q)
	\mathrm{d}\mu_{\nu_3,\gamma_3-\nu_3}(w;q).
\end{align}
Here and in what follows, we define the $q$-analogue of the Dirichlet measure \eqref{DirichletMeasure} as follows:
\begin{equation}\label{q-DirichletMeasure}
	\mathrm{d}\mu_{\alpha,\beta}(t;q):= \mathfrak{m}_{\alpha,\beta}(t;q)\mathrm{d}_q t,
\end{equation}
where
\[
\mathfrak{m}_{\alpha,\beta}(t;q):=\frac{\Gamma_q(\alpha+\beta)}{\Gamma_q(\alpha) \Gamma_q(\beta)} t^{\alpha-1} \frac{(tq;q)_{\infty}}{(tq^{\beta};q)_{\infty}}, ~~~\min\{\Re(\alpha),\Re(\beta)\}>0.
\]
Obviously, we have $\displaystyle\int_{0}^{1}\mathrm{d}\mu_{\alpha,\beta}(t;q)=1$ and $\displaystyle\lim_{q\rightarrow1^{-}}\mathfrak{m}_{\alpha,\beta}(t;q)=\mathfrak{m}_{\alpha,\beta}(t)$.

Actually, in Section \ref{Section-5}, we shall first prove a Joshi-Vyas type theorem and derive some useful corollaries from it. Then a discrete analogue for one of the corollaries is obtained. Finally, we present a $q$-analogue of Theorem \ref{Th-ErdelyiIntegral-FK}. 

In the next section, we will further introduce some results related to $q$-series.

\section{Preliminaries}\label{Section-2} 
 
Gasper \cite{Gasper-2000} discussed the $q$-analogues of Erd\'{e}lyi's integrals \eqref{ErdelyiIntegral-1}--\eqref{ErdelyiIntegral-3} and provided their discrete analogues. The integral \eqref{ErdelyiIntegral-1} has the following $q$-analogue \cite[p. 4, Eq. (1.13)]{Gasper-2000}
\begin{align} \label{eq:2phi1 Erdelyi1}
	{}_2\widetilde{\phi}_1 \left[ \begin{matrix}
		\alpha, \beta \\
		\gamma
	\end{matrix}; q, x \right]
	&=\int_0^1 \frac{(xtq^{\alpha'}; q)_{\infty}}{(xt; q)_{\infty}} {}_2\widetilde{\phi}_1 \left[\begin{matrix}
		\alpha-\alpha', \beta \\
		\lambda
	\end{matrix}; q, xtq^{\alpha'} \right]\notag\\
	&\hspace{0.5cm}\cdot {}_3 \phi_2 \left[ \begin{matrix}
			q^{\alpha'}, q^{\beta-\lambda}, t^{-1} \\
			q^{\gamma-\lambda}, q/(xt)
		\end{matrix};q,q\right] \mathrm{d}\mu_{\lambda, \gamma-\lambda}(t;q),
\end{align}
where $\Re(\gamma)>\Re(\lambda)>0$ and $\mathrm{d}\mu_{\alpha,\beta}(t;q)$ is defined in \eqref{q-DirichletMeasure}. Before stating the $q$-analogue of \eqref{ErdelyiIntegral-3}, we have to first introduce the $q$-analogue of the hypergeometric measure \eqref{HypergeometricMeasure}: 
\begin{equation}\label{q-HypergeometricMeasure}
	\mathrm{d}\mu_{\alpha,\beta,\gamma,\eta}(t;q)
	:=\mathfrak{m}_{\alpha,\beta,\gamma,\eta}(t;q)\mathrm{d}_q t,
\end{equation}
where $\min\{\Re(\eta), \Re(\gamma), \Re(\eta+\gamma-\alpha-\beta)\}>0$ and
\begin{align}
	\mathfrak{m}_{\alpha,\beta,\gamma,\eta}(t;q)
	&=\frac{\Gamma_{q}(\eta+\gamma-\alpha)\Gamma_q(\eta+\gamma-\beta)}{\Gamma_q(\eta)\Gamma_q(\gamma)\Gamma_q(\eta+\gamma-\alpha-\beta)}\notag\\
	&\hspace{1cm}\cdot t^{\eta-1}\frac{(tq;q)_{\infty}}{(tq^{\gamma};q)_\infty}
	{}_{3}\phi_1\left[\begin{matrix}
		q^{\alpha}, q^{\beta},t^{-1}\\
		q^{\gamma}
	\end{matrix};q,tq^{\gamma-\alpha-\beta}\right].
\end{align} 
We have $\displaystyle 
\int_0^1\mathrm{d}\mu_{\eta-\lambda,\gamma-\lambda,\gamma-\lambda+\eta-\nu,\nu}(t;q)=1$ and $\displaystyle 
\lim_{q\rightarrow 1^{-}}\mathfrak{m}_{\alpha,\beta,\gamma,\eta}(t;q)
=\mathfrak{m}_{\alpha,\beta,\gamma,\eta}(t)$.
Then the $q$-analogue of \eqref{ErdelyiIntegral-3} is given by \cite[p. 4, Eq. (1.14)]{Gasper-2000}
\begin{align} \label{eq:2phi1 Erdelyi2}
	{}_2\widetilde{\phi}_1\left[\begin{matrix}
		\alpha,\beta \\
		\gamma
	\end{matrix}; q, x \right]
	&=\int_0^1{}_3\widetilde{\phi}_2\left[\begin{matrix}
		\alpha, \beta,\eta \\
		\lambda, \nu
	\end{matrix};q,xt\right]\mathrm{d}\mu_{\eta-\lambda,\gamma-\lambda,\gamma-\lambda+\eta-\nu,\nu}(t;q).
\end{align}
where $\min\{\Re(\lambda), \Re(\nu), \Re(\gamma+\eta-\lambda-\nu)\}>0$.  Joshi and Vyas \cite{Joshi-Vyas-2006} also derived some nice $q$-Erd\'{e}lyi-type integrals, one of which is discussed in Section \ref{Section-5}.

In the same paper \cite{Gasper-2000}, Gasper also found the following discrete $q$-analogue of \eqref{eq:2phi1 Erdelyi2}:
\begin{align}\label{Gasper-DiscreteAnalog}
	{}_{3}\phi_{2}\left[\begin{matrix}
		\alpha,\beta,q^{-n}\\
		\gamma,\delta 
	\end{matrix};q,q\right]
	&=\frac{(q,\lambda;q)_n}{(\gamma,\mu;q)_n}\sum_{k=0}^{n}
	\frac{(\nu;q)_k(\gamma\mu/(\lambda\nu);q)_{n-k}}{(q;q)_k(q;q)_{n-k}}
	\nu^{n-k}\notag\\
	&\cdot {}_{3}\phi_{2}\left[\begin{matrix}
		\mu/\lambda,\gamma/\lambda,q^{k-n}\\
		\gamma\mu/(\lambda\nu), q^{1-n}/\lambda 
	\end{matrix};q,q^{1-k}/\nu\right]
	{}_{4}\phi_{3}\left[\begin{matrix}
		\alpha, \beta, \mu, q^{-k}\\
		\lambda, \nu, \delta 
	\end{matrix};q,q\right].
\end{align}

Finally, we also need the following $q$-hypergeometric function of three variables (see \cite[p. 96]{Denis-1988} and \cite[p. 35]{Rassias-Singh-Srivastava-1994}): 
\begin{align}\label{Def-q-TripleSeries}
	\phi^{(3)}[x,y,z]&\equiv\phi^{(3)} \left[ \begin{matrix}
			(a) :: (b); (b'); (b''): (c); (c'); (c'') \\
			(e) :: (g); (g'); (g''): (h); (h'); (h'')
	\end{matrix}; q, x, y, z  \right]\notag\\
	&:=\sum_{m,n,p=0}^{\infty} 
	\frac{\displaystyle \prod_{r=1}^A (a_r;q)_{m+n+p}}{\displaystyle \prod_{r=1}^E (e_r;q)_{m+n+p} }\cdot 
	\frac{\displaystyle
		\prod_{r=1}^B (b_r;q)_{m+n} 
		\prod_{r=1}^{B'} (b_r';q)_{n+p} 
		\prod_{r=1}^{B''} (b_r'';q)_{p+m}}{\displaystyle 
		\prod_{r=1}^G (g_r;q)_{m+n} 
		\prod_{r=1}^{G'} (g_r';q)_{n+p} 
		\prod_{r=1}^{G''} (g_r'';q)_{p+m}}\notag\\
	&\hspace{0.5cm}\cdot\frac{\displaystyle \prod_{r=1}^C (c_r;q)_{m}\prod_{r=1}^{C'} (c_r';q)_{n} \prod_{r=1}^{C''} (c_r'';q)_{p}}{\displaystyle \prod_{r=1}^H (h_r;q)_{m}\prod_{r=1}^{H'} (h_r';q)_{n} \prod_{r=1}^{H''} (h_r'';q)_{p}}
	\cdot\frac{x^m}{(q;q)_m}\frac{y^n}{(q;q)_n}\frac{z^p}{(q;q)_p},
\end{align}
where $|x|<1$, $|y|<1$ and $|z|<1$.

\section{Theorem \ref{Th-ErdelyiIntegral-FK} revisited}
\subsection{An alternative proof of Theorem \ref{Th-ErdelyiIntegral-FK}}\label{Section-3}

For convenience, let $I$ denote the triple integral in \eqref{Th-ErdelyiIntegral-FK-1}. By \eqref{Def-FK-2F1}, the first $F_K$-function in the integrand can be expanded as
\begin{align}\label{Th-ErdelyiIntegral-FK-Proof-1}
	&F_K[ \alpha_1, \alpha_2-\eta_2, \alpha_2-\eta_2, \beta_1-\lambda_3, \beta_2, \beta_1-\lambda_3; \alpha_1-\lambda_1+\eta_1, \beta_2-\lambda_2+\mu_2, \beta_1-\lambda_3; ux, vy, wz]\notag\\
	&\hspace{1cm}=\sum_{m=0}^{\infty}\frac{(\alpha_2-\eta_2)_m}{m!} 
	{}_2F_1 \left[\begin{matrix}
		\beta_1-\lambda_3+m, \alpha_1 \\
		\alpha_1-\lambda_1+\eta_1
	\end{matrix}; ux
	\right] 
	{}_2F_1 \left[\begin{matrix}
		\alpha_2-\eta_2+m, \beta_2 \\
		\beta_2-\lambda_2+\mu_2
	\end{matrix}; vy
	\right] (wz)^m.
\end{align}
Since $\mathbb{D}_{K}$ is a Reinhardt domain, we have immediately $(ux,vy,wz)\in \mathbb{D}_{K}$ $(u,v,w\in(0,1))$ and thus the region of convergence of the series in \eqref{Th-ErdelyiIntegral-FK-Proof-1} is clear. For the second $F_K$-function in the integrand, we have 
\begin{align}\label{Th-ErdelyiIntegral-FK-Proof-2}
	& F_K \left[\lambda_1-\eta_1, \eta_2, \eta_2, \lambda_3, \lambda_2-\mu_2, \lambda_3; \lambda_1, \lambda_2, \lambda_3; \frac{(1-u)x}{1-ux}, \frac{(1-v)y}{1-vy}, \frac{wz}{(1-ux)(1-vy)} \right]\notag\\
	&\hspace{0.5cm}=\sum_{n=0}^{\infty} \frac{(\eta_2)_n}{n!} 
	{}_2F_1 \left[\begin{matrix}
		\lambda_3+n, \lambda_1-\eta_1 \\
		\lambda_1
	\end{matrix}; \frac{(1-u)x}{1-ux} \right]\notag\\
	&\hspace{1.5cm}\cdot  
	{}_2F_1 \left[\begin{matrix}
		\eta_2+n, \lambda_2-\mu_2 \\
		\lambda_2
	\end{matrix}; \frac{(1-v)y}{1-vy}\right]
	\cdot\left[\frac{wz}{(1-ux)(1-vy)}\right]^n.
\end{align}
Note that
\[
\left|\frac{(1-u)x}{1-ux}\right|
<\frac{1-u}{1-u|x|}<1, ~~~
\left|\frac{(1-v)x}{1-vx}\right|
<\frac{1-v}{1-v|x|}<1
\]
and
\begin{align*}
\frac{|\frac{wz}{(1-ux)(1-vy)}|}{\big(1-\big|\frac{(1-u)x}{1-ux}\big|\big)\big(1-\big|\frac{(1-v)y}{1-vy}\big|\big)}
&=\frac{w|z|}{(|1-ux|-(1-u)|x|)(|1-vy|-(1-v)|y|)}\\
&\leq\frac{|z|}{(1-|x|)(1-|y|)}<1,
\end{align*}
namely, 
\[
\left(\frac{(1-u)x}{1-ux},\frac{(1-v)y}{1-vy},\frac{wz}{(1-ux)(1-vy)}\right)\in\mathbb{D}_{K}.
\]
Thus the region of convergence of the series in \eqref{Th-ErdelyiIntegral-FK-Proof-2} is also fully characterized. 

Now we have 
\begin{align}\label{Th-ErdelyiIntegral-FK-Proof-3}
	I&=\sum_{m=0}^{\infty}\frac{(\alpha_2-\eta_2)_m}{m!}z^m
	\sum_{n=0}^{\infty} \frac{(\eta_2)_n}{n!}z^n
	\int_{0}^{1} w^{m+n}\mathrm{d} \mu_{\beta_1, \gamma_3-\beta_1}(w)\notag\\
	&\hspace{0.5cm}\cdot 
	\int_{0}^{1}(1-ux)^{-\lambda_3-n}
	{}_2F_1 \left[\begin{matrix}
		\beta_1-\lambda_3+m, \alpha_1 \\
		\alpha_1-\lambda_1+\eta_1
	\end{matrix}; ux
	\right] 
	{}_2F_1 \left[\begin{matrix}
		\lambda_3+n, \lambda_1-\eta_1 \\
		\lambda_1
	\end{matrix}; \frac{(1-u)x}{1-ux} \right]\mathrm{d} \mu_{\alpha_1-\lambda_1+\eta_1, \lambda_1} (u)\notag\\
	&\hspace{0.5cm}\cdot
	\int_{0}^{1}(1-vy)^{-\eta_2-n}
	{}_2F_1 \left[\begin{matrix}
		\alpha_2-\eta_2+m, \beta_2 \\
		\beta_2-\lambda_2+\mu_2
	\end{matrix}; vy
	\right]{}_2F_1 \left[\begin{matrix}
		\eta_2+n, \lambda_2-\mu_2 \\
		\lambda_2
	\end{matrix}; \frac{(1-v)y}{1-vy}\right]\mathrm{d} \mu_{\beta_2-\lambda_2+\mu_2, \lambda_2} (v) .
\end{align}
Note that
\begin{equation}\label{Th-ErdelyiIntegral-FK-Proof-4}
	\int_{0}^{1}w^{m+n}\mathrm{d}\mu_{\beta_1, \gamma_3-\beta_1}(w)=\frac{(\beta_1)_{m+n}}{(\gamma_3)_{m+n}},
\end{equation}
where $\Re(\gamma_3)>\Re(\beta_1)>0$. By letting $\alpha'\rightarrow\lambda_3+n$, $\beta\rightarrow\alpha_1$, $\lambda\rightarrow\alpha_1-\lambda_1+\eta_1$, $\alpha\rightarrow\beta_1+m+n$ and $\gamma\rightarrow\alpha_1+\eta_1$ in the Erd\'{e}lyi integral \eqref{ErdelyiIntegral-1}, we obtain
\begin{align}\label{Th-ErdelyiIntegral-FK-Proof-5}
&\int_{0}^{1}(1-ux)^{-\lambda_3-n}
{}_2F_1 \left[\begin{matrix}
	\beta_1-\lambda_3+m, \alpha_1 \\
	\alpha_1-\lambda_1+\eta_1
\end{matrix}; ux
\right] \notag\\
&\hspace{0.5cm}\cdot {}_2F_1 \left[\begin{matrix}
	\lambda_3+n, \lambda_1-\eta_1 \\
	\lambda_1
\end{matrix}; \frac{(1-u)x}{1-ux} \right]\mathrm{d} \mu_{\alpha_1-\lambda_1+\eta_1, \lambda_1} (u)={}_2F_1 \left[\begin{matrix}
	\beta_1+m+n, \alpha_1 \\
	\alpha_1+\eta_1
\end{matrix}; x
\right],
\end{align}
where $\Re(\alpha_1+\eta_1)>\Re(\lambda_1)>0$. 
By letting $\alpha'\rightarrow\eta_2+n$, $\beta\rightarrow\beta_2$, 
$\lambda\rightarrow\beta_2-\lambda_2+\mu_2$,  $\alpha\rightarrow\alpha_2+m+n$ and $\gamma\rightarrow\beta_2+\mu_2$ in the Erd\'{e}lyi integral \eqref{ErdelyiIntegral-1}, we obtain
\begin{align}\label{Th-ErdelyiIntegral-FK-Proof-6}
	&\int_{0}^{1}(1-vy)^{-\eta_2-n}
	{}_2F_1 \left[\begin{matrix}
		\alpha_2-\eta_2+m, \beta_2 \\
		\beta_2-\lambda_2+\mu_2
	\end{matrix}; vy
	\right]\notag\\
	&\hspace{0.5cm}\cdot{}_2F_1 \left[\begin{matrix}
		\eta_2+n, \lambda_2-\mu_2 \\
		\lambda_2
	\end{matrix}; \frac{(1-v)y}{1-vy}\right]\mathrm{d} \mu_{\beta_2-\lambda_2+\mu_2, \lambda_2} (v)={}_2F_1 \left[\begin{matrix}
	\alpha_2+m+n, \beta_2 \\
	\beta_2+\mu_2
	\end{matrix};y
	\right],
\end{align}
where $\Re(\beta_2+\mu_2)>\Re(\lambda_2)>0$. 
Substituting \eqref{Th-ErdelyiIntegral-FK-Proof-4}, \eqref{Th-ErdelyiIntegral-FK-Proof-5} and \eqref{Th-ErdelyiIntegral-FK-Proof-6} in \eqref{Th-ErdelyiIntegral-FK-Proof-3} gives
\begin{align*}
	I&=\sum_{m,n=0}^{\infty}
	\frac{(\beta_1)_{m+n}(\alpha_2-\eta_2)_m(\eta_2)_n}{(\gamma_3)_{m+n}m!n!}
	{}_2F_1 \left[\begin{matrix}
		\beta_1+m+n, \alpha_1 \\
		\alpha_1+\eta_1
	\end{matrix}; x
	\right]
	{}_2F_1 \left[\begin{matrix}
		\alpha_2+m+n, \beta_2 \\
		\beta_2+\mu_2
	\end{matrix};y
	\right]z^{m+n}\\
	&=\sum_{m=0}^{\infty}
	\frac{(\beta_1)_m}{(\gamma_3)_m}
	{}_2F_1 \left[\begin{matrix}
		\beta_1+m, \alpha_1 \\
		\alpha_1+\eta_1
	\end{matrix}; x
	\right]
	{}_2F_1 \left[\begin{matrix}
		\alpha_2+m, \beta_2 \\
		\beta_2+\mu_2
	\end{matrix};y
	\right]z^{m}\sum_{n=0}^{m}
	\frac{(\alpha_2-\eta_2)_{m-n}(\eta_2)_n}{(m-n)!n!}.
\end{align*}
The result \eqref{Th-ErdelyiIntegral-FK-1} now follows by applying the Chu-Vandermonde identity
\[
(a+b)_{m}=\sum_{n=0}^{m}\binom{m}{n}(a)_{m-n}(b)_{n}.
\]

\begin{remark}
	Joshi and Vyas \cite{Joshi-Vyas-2003} have shown that the Erd\'{e}lyi integrals \eqref{ErdelyiIntegral-1}-\eqref{ErdelyiIntegral-3} can be derived by using the series manipulation technique. Thus, inspired by the proof presented in this section, we can also write down a proof using only the series manipulation technique. However, such a proof is quite complicated, so we choose to omit it here.  
\end{remark}

\subsection{A curious integral representation for $F_2$}\label{Section-4}

In this subsection, we derive a neat and interesting integral from Theorem \ref{Th-ErdelyiIntegral-FK}. This integral is novel and has not been mentioned in the existing literature.

\begin{theorem}\label{Th-CuriousIntegral}
	Let $\Re(c_1)>\Re(d_1)>0$ and $\Re(c_2)>\Re(b_2)>0$. Then we have
	\begin{align}\label{Th-CuriousIntegral-1}
		F_{2}[a_1,b_1,b_2;c_1,c_2;y,z]
		&=\int_{0}^{1}\int_{0}^{1}(1-wz-vy)^{-a_1}
		{}_{2}F_{1}\left[\begin{matrix}
			a_1-a_2, c_1-b_1-d_1\\
			c_1-d_1
		\end{matrix};\frac{vy}{vy+wz-1}\right]\notag\\
		&\hspace{0.5cm}\cdot{}_{2}F_{1}\left[\begin{matrix}
			a_2, b_1+d_1-c_1\\
			d_1
		\end{matrix};\frac{(1-v)y}{1-vy-wz}\right]
		\mathrm{d}\mu_{c_1-d_1,d_1}(v)
		\mathrm{d}\mu_{b_2,c_2-b_2}(w).
	\end{align}
\end{theorem}
\begin{proof}
	Letting $x=0$ in \eqref{Th-ErdelyiIntegral-FK-1}, it reduces to
	\begin{align}\label{Th-CuriousIntegral-Proof-1}
		& F_2[\alpha_2, \beta_2, \beta_1; \beta_2+\mu_2, \gamma_3; y,z]\notag\\
		&\hspace{0.5cm} = \int_0^1\int_0^1 (1-vy)^{-\eta_2} F_2[\alpha_2-\eta_2, \beta_2, \beta_1-\lambda_3; \beta_2-\lambda_2+\mu_2, \beta_1-\lambda_3; vy, wz]\notag\\
		&\hspace{1.5cm}\cdot F_2 \left[\eta_2, \lambda_2-\mu_2, \lambda_3; \lambda_2, \lambda_3; \frac{(1-v)y}{1-vy}, \frac{wz}{1-vy} \right]\mathrm{d} \mu_{\beta_2-\lambda_2+\mu_2, \lambda_2} (v) \mathrm{d} \mu_{\beta_1, \gamma_3-\beta_1}(w).
	\end{align}
	Next, letting $\alpha_2\rightarrow a_1$, $\beta_2\rightarrow b_1$, $\beta_1\rightarrow b_2$, $\mu_2\rightarrow c_1-b_1$, $\gamma_3\rightarrow c_2$, $\lambda_2\rightarrow d_1$, $\lambda_3\rightarrow d_2$ and $\eta_2\rightarrow a_2$ in \eqref{Th-CuriousIntegral-Proof-1} gives   
	\begin{align}\label{Th-CuriousIntegral-Proof-2}
		& F_2[a_1, b_1, b_2; c_1, c_2; y,z]\notag\\
		&\hspace{0.5cm} = \int_0^1\int_0^1 (1-vy)^{-a_2} F_2[a_1-a_2, b_1, b_2-d_2; c_1-d_1, b_2-d_2; vy, wz]\notag\\
		&\hspace{1.5cm}\cdot F_2 \left[a_2, d_1-c_1+b_1, d_2; d_1, d_2; \frac{(1-v)y}{1-vy}, \frac{wz}{1-vy} \right]\mathrm{d} \mu_{c_1-d_1, d_1} (v) \mathrm{d} \mu_{b_2, c_2-b_2}(w).
	\end{align}
	Then, by using \cite[p. 306, Eq. (109)]{Srivastava-Karlsson-Book-1985} 
	\begin{equation}\label{Th-CuriousIntegral-Proof-3}
		F_2[a, b, b'; c, b'; y,z]=(1-z)^{-a}
		{}_{2}F_{1}\left[\begin{matrix}
			a, b\\
			c
		\end{matrix};\frac{y}{1-z}\right].
	\end{equation}
	we have
	\begin{align}\label{Th-CuriousIntegral-Proof-4}
	&(1-vy)^{-a_2}F_2 \left[a_2, d_1-c_1+b_1, d_2; d_1, d_2; \frac{(1-v)y}{1-vy}, \frac{wz}{1-vy} \right]\notag\\
	&\hspace{1cm}=(1-vy-wz)^{-a_2}{}_{2}F_{1}\left[\begin{matrix}
		a_2, b_1+d_1-c_1\\
		d_1
	\end{matrix};\frac{(1-v)y}{1-vy-wz}\right].
	\end{align}
	By using \eqref{Th-CuriousIntegral-Proof-3} and the Pfaff transformation \cite[p. 300, Eq. (79)]{Srivastava-Karlsson-Book-1985}, we have
	\begin{align}\label{Th-CuriousIntegral-Proof-5}
		&F_2[a_1-a_2, b_1, b_2-d_2; c_1-d_1, b_2-d_2; vy, wz]\notag\\
		&\hspace{1cm}=(1-wz)^{-a_1+a_2}{}_{2}F_{1}\left[\begin{matrix}
			a_1-a_2, b_1\\
			c_1-d_1
		\end{matrix};\frac{vy}{1-wz}\right]\notag\\
		&\hspace{1cm}=(1-wz-vy)^{-a_1+a_2}{}_{2}F_{1}\left[\begin{matrix}
			a_1-a_2, c_1-b_1-d_1\\
			c_1-d_1
		\end{matrix};\frac{vy}{vy+wz-1}\right].
	\end{align}
	Substituting \eqref{Th-CuriousIntegral-Proof-4} and \eqref{Th-CuriousIntegral-Proof-5} into \eqref{Th-CuriousIntegral-Proof-2} leads us to \eqref{Th-CuriousIntegral-1}. 
\end{proof}
\begin{remark}
	We can prove \eqref{Th-CuriousIntegral-Proof-2} directly by using the method described in Section \ref{Section-3}. In addition, the integral \eqref{Th-CuriousIntegral-1} can also be derived from Manocha's integral \cite[p. 239, Eq. (13)]{Manocha-1967}:
	\begin{align}\label{ManochaIntegral}
		&F_2[a,b,c;d,e;y,z]
		=\int_{0}^{1}\int_{0}^{1}(1-vy-wz)^{-a'}
		F_2[a-a',b,c;\lambda,\eta;yv,zw]\notag\\
		&\hspace{1cm} \cdot F_2\left[a',b-\lambda,c-\eta;d-\lambda,e-\eta;\frac{(1-v)y}{1-vx-wz},\frac{(1-w)z}{1-vy-wz}\right]\mathrm{d}\mu_{\lambda,d-\lambda}(v)\mathrm{d}\mu_{\eta,e-\eta}(w),
	\end{align}
	where $\Re(d)>\Re(\lambda)>0$ and $\Re(e)>\Re(\eta)>0$. 
	In fact, if we set $\eta\rightarrow c$ in \eqref{ManochaIntegral} and then use the two transformation formulas as in the proof above, we obtain
	\begin{align}\label{ManochaIntegral-1}
		&F_2[a,b,c;d,e;y,z]
		=\int_{0}^{1}\int_{0}^{1}(1-vy-wz)^{-a}
		{}_{2}F_{1}\left[\begin{matrix}
			a-a',\lambda-b\\
			\lambda 
		\end{matrix};\frac{vy}{vy+wz-1}\right]
		\notag\\
		&\hspace{1cm} \cdot {}_{2}F_{1}\left[
		\begin{matrix}
			a',b-\lambda\\
			d-\lambda
		\end{matrix};\frac{(1-v)y}{1-vy-wz}\right]\mathrm{d}\mu_{\lambda,d-\lambda}(v)\mathrm{d}\mu_{c,e-c}(w).
	\end{align}
	Letting further $\lambda\rightarrow c_1-d_1$, $b\rightarrow b_1$, $a\rightarrow a_1$, $a'\rightarrow a_2$, $c\rightarrow b_2$, $d\rightarrow c_1$ and $e\rightarrow c_2$ in \eqref{ManochaIntegral-1} leads us again to the integral \eqref{Th-CuriousIntegral-1}.  
\end{remark}

\subsection{A further extension of Theorem \ref{Th-ErdelyiIntegral-FK}}\label{Section-3-3}

As usual, the discrete convolution product of two sequences $\mathbf{a}:=a(m,n)$ and $\mathbf{b}:=b(m,n)$ is defined by
\[
(\mathbf{a}\star\mathbf{b})(m,n)
:=\sum_{i=0}^{m}\sum_{j=0}^{n}a(m-i,n-j)b(i,j). 
\]
Let us consider the function of the following form:
\begin{align}\label{Section-3-3-Def}
	\mathcal{F}^{\mathbf{a}}\left[\begin{matrix}
		\alpha_1,\beta_1:\alpha_2,\beta_2\\
		\gamma_1:\gamma_2
	\end{matrix};x_1,x_2,x_3,x_4\right]
	&:=\sum_{m,n=0}^{\infty}
	a(m,n){}_{2}F_{1}\left[\begin{matrix}
		\alpha_1+m,\beta_1\\
		\gamma_1 
	\end{matrix};x_1\right]\notag \\
	&\hspace{0.5cm} \cdot
	{}_{2}F_{1}\left[\begin{matrix}
		\alpha_2+n,\beta_2\\
		\gamma_2
	\end{matrix};x_2\right]x_3^{m}x_{4}^{n}.
\end{align}
Functions of this type appear frequently in many situations and have  properties similar to Saran's $F_K$-function. 

If we take
\[
a(m,n)=
\begin{cases}
	\displaystyle \frac{(\alpha_1)_n(\alpha_2)_n}{ n!(\gamma_3)_n}, & m=n,\\
	0, & m\neq n,
\end{cases}
\] 
then the function \eqref{Section-3-3-Def} reduces to Saran's function:
\[
F_K\left[\beta_1,\alpha_2,\alpha_2,\alpha_1,\beta_2,\alpha_1;\gamma_1,\gamma_2,\gamma_3;x_1,x_2,x_3x_4\right].
\]

Another significant example originates in physics. In 2018, Belitsky \cite[p. 69]{Belitsky-2018} generalized Saran's $F_K$-function of three variables to $L$ variables, i.e., 
\begin{align}\label{GeneralizedSaranFK}
	&F_K\left[\begin{matrix}
		a_1, b_1, \cdots, b_{L-1}, a_2\\
		c_1, \cdots, c_{L}
	\end{matrix}; z_1,\cdots, z_L\right]\notag\\
	&\hspace{1cm}:=\sum_{n_1,\cdots, n_L=0}^{\infty}
	\frac{\left(a_1\right)_{n_1}\left(b_1\right)_{n_1+n_2}\cdots \left(b_{L-1}\right)_{n_{L-1}+n_{L}}\left(a_2\right)_{n_L}}{\left(c_1\right)_{n_1}\cdots\left(c_L\right)_{n_L}}
	\frac{z_1^{n_1}}{n_1!}
	\cdots
	\frac{z_{L}^{n_L}}{n_L!}.
\end{align}
Then Rosenhaus \cite{Rosenhaus-2019} employed the function \eqref{GeneralizedSaranFK} in his study of $n$-point conformal blocks in the comb channel. More recently, Ferrando \emph{et al.} \cite{Ferrando-Loebbert-Pitters-Stawinski-2025} used a \emph{rescaled version} of this function in evaluating the conformal integral associated with the $n$-point conformal partial wave in the comb channel. Taking $L=4$ in \eqref{GeneralizedSaranFK}, we obtain
\begin{align}\label{GeneralizedSaranFK-L4}
	&F_K\left[\begin{matrix}
		a_1, b_1, b_2, b_3, a_2\\
		c_1, c_2, c_3, c_4
	\end{matrix}; z_1, z_2, z_3, z_4\right]\notag\\
	&\hspace{0.5cm}=\sum_{n_1,n_2,n_3,n_4=0}^{\infty}
	\frac{(a_1)_{n_1}(b_1)_{n_1+n_2}(b_2)_{n_2+n_3} (b_3)_{n_3+n_4}(a_2)_{n_4}}{(c_1)_{n_1}(c_2)_{n_2}(c_3)_{n_3}(c_4)_{n_4}}
	\frac{z_1^{n_1}}{n_1!}
	\frac{z_2^{n_2}}{n_2!}
	\frac{z_3^{n_3}}{n_3!}
	\frac{z_4^{n_4}}{n_4!}\notag\\
	&\hspace{0.5cm}=
	\sum_{n_2,n_3=0}^{\infty}
	\frac{(b_1)_{n_2}(b_2)_{n_2+n_3}(b_3)_{n_3}}{(c_2)_{n_2}(c_3)_{n_3}n_2!n_3!} 
	{}_{2}F_{1}\left[\begin{matrix}
		a_1,b_1+n_2\\
		c_1
	\end{matrix};z_1\right]{}_{2}F_{1}\left[\begin{matrix}
	a_2,b_3+n_3\\
	c_4
	\end{matrix};z_4\right]
	z_2^{n_2}
	z_3^{n_3},
\end{align}
which also has the form of \eqref{Section-3-3-Def}. The region of convergence of \eqref{GeneralizedSaranFK-L4} is determined by
\[
|z_1|<1,~|z_4|<1~~~\text{and}~~~
\frac{|z_2|}{1-|z_1|}+\frac{|z_3|}{1-|z_4|}<1.
\]
Here it is worth mentioning that Khichi has introduced the multivariate hypergeometric functions $H_B^{(n)}$ (\cite[p. 308, Eq. (123)]{Srivastava-Karlsson-Book-1985}), which appears quite similar to the function \eqref{GeneralizedSaranFK}, although they are different in nature.

By using the method from Section \ref{Section-3}, we have the following theorem. 
\begin{theorem}\label{Th-Section-3-3}
	Let $\Re(\tau_j)>\Re(\gamma_j)>0$ $(j=1,\cdots,4)$. Then we have
	\begin{align}\label{Th-Section-3-3-1}
		&\mathcal{F}^{\mathbf{c}}\left[\begin{matrix}
			\alpha_1+\lambda_1,\beta_1:\alpha_2+\lambda_2,\beta_2\\
			\tau_1:\tau_2
		\end{matrix};x_1,x_2,x_3,x_4\right]
		=\int_0^1\int_0^1\int_0^1\int_0^1
		(1-u_1 x_1)^{-\lambda_1}
		(1-u_2 x_2)^{-\lambda_2}\notag\\
		&\hspace{0.5cm} \cdot\mathcal{F}^{\mathbf{a}}\left[\begin{matrix}
			\alpha_1,\beta_1:\alpha_2,\beta_2\\
			\gamma_1:\gamma_2
		\end{matrix};x_1 u_1,x_2 u_2,x_3 u_3,x_4 u_4\right]\notag\\
		&\hspace{0.5cm} \cdot\mathcal{F}^{\mathbf{b}}\left[\begin{matrix}
			\lambda_1,\beta_1-\gamma_1:\lambda_2,\beta_2-\gamma_2\\
			\tau_1-\gamma_1:\tau_2-\gamma_2 
		\end{matrix};\frac{(1-u_1)x_1}{1-u_1 x_1},\frac{(1-u_2)x_2}{1-u_2 x_2},\frac{u_3 x_3}{1-u_1 x_1},\frac{u_4 x_4}{1-u_2 x_2}\right]
		\prod_{j=1}^{4}\mathrm{d}\mu_{\gamma_j,\tau_j-\gamma_j}(u_j),
	\end{align}
	where
	\[
	\mathbf{c}(m,n)=\frac{(\gamma_3)_{m}}{(\tau_3)_{m}}\frac{(\gamma_4)_{n}}{(\tau_4)_{n}}(\mathbf{a}\star\mathbf{b})(m,n). 
	\]
\end{theorem}
\begin{proof}
	Since the proof is very similar to the one given in Section \ref{Section-3}, we omit the details and retain only key steps. 
	Let us denote the right-hand side of \eqref{Th-Section-3-3-1} by $I$. Then, by using \eqref{Section-3-3-Def}, we have 
	\begin{align}
		I&=\sum_{m_1,m_2=0}^{\infty}\sum_{n_1,n_2=0}^{\infty}
		a(m_1,m_2)b(n_1,n_2)x_3^{m_1+n_1}x_4^{m_2+n_2}\notag\\ 
		&\hspace{0.5cm}\cdot
		\int_{0}^{1}(1-u_1 x_1)^{-\lambda_1-n_1}
		{}_{2}F_{1}\left[\begin{matrix}
			\alpha_1+m_1,\beta_1\\
			\gamma_1 
		\end{matrix};x_1 u_1\right]
		{}_{2}F_{1}\left[\begin{matrix}
			\lambda_1+n_1,\beta_1-\gamma_1\\
			\tau_1-\gamma_1
		\end{matrix};\frac{(1-u_1)x_1}{1-u_1 x_1}\right]\mathrm{d}\mu_{\gamma_1,\tau_1-\gamma_1}(u_1)\notag\\
		&\hspace{0.5cm}\cdot 
		\int_{0}^{1}(1-u_2 x_2)^{-\lambda_2-n_2} {}_{2}F_{1}\left[\begin{matrix}
			\alpha_2+m_2,\beta_2\\
			\gamma_2 
		\end{matrix};x_2 u_2\right]
		{}_{2}F_{1}\left[\begin{matrix}
			\lambda_2+n_2,\beta_2-\gamma_2\\
			\tau_2-\gamma_2
		\end{matrix};\frac{(1-u_2)x_2}{1-u_2 x_2}\right]\mathrm{d}\mu_{\gamma_2,\tau_2-\gamma_2}(u_2)\notag\\
		&\hspace{0.5cm}\cdot \int_{0}^{1} u_3^{m_1+n_1}\mathrm{d}\mu_{\gamma_3,\tau_3-\gamma_3}(u_3)
		\int_{0}^{1} u_4^{m_2+n_2}\mathrm{d}\mu_{\gamma_4,\tau_4-\gamma_4}(u_4).
	\end{align}
	By using \eqref{ErdelyiIntegral-1} and \eqref{Th-ErdelyiIntegral-FK-Proof-4}, we obtain
	\begin{align*}
		I&=\sum_{m_1,m_2=0}^{\infty}\sum_{n_1,n_2=0}^{\infty}
		a(m_1,m_2)b(n_1,n_2)x_3^{m_1+n_1}x_4^{m_2+n_2}\frac{(\gamma_3)_{m_1+n_1}}{(\tau_3)_{m_1+n_1}}\frac{(\gamma_4)_{m_2+n_2}}{(\tau_4)_{m_2+n_2}}\\
		&\hspace{1cm}\cdot {}_{2}F_{1}\left[\begin{matrix}
			\alpha_1+\lambda_1+m_1+n_1,\beta_1\\
			\tau_1 
		\end{matrix};x_1\right]
		{}_{2}F_{1}\left[\begin{matrix}
			\alpha_2+\lambda_2+m_2+n_2,\beta_2\\
			\tau_2
		\end{matrix};x_2\right],
	\end{align*}
	which, through the same series manipulation as in Section \ref{Section-3}, leads us to the desired result. 
\end{proof}

\section{Some $q$-Erd\'{e}lyi-type integrals}\label{Section-5}

\subsection{A Joshi-Vyas type theorem}

Inspired by the formula \cite[p. 133]{Joshi-Vyas-2003} of Joshi and Vyas and its useful $q$-analogue \cite[p. 646]{Joshi-Vyas-2006}, we establish--using the measure in \eqref{q-HypergeometricMeasure}--the following theorem of Joshi-Vyas type.

\begin{theorem}\label{Th-JoshVyas-generalization}
	Let $\Phi(z_1,\cdots,z_k)$ be defined by
	\[
	\Phi(z_1,\cdots,z_k):=\sum_{n_1,\cdots,n_k=0}^{\infty}c(n_1,\cdots,n_k)z_1^{n_1}\cdots z_k^{n_k}
	\]
	where $\{c(n_1,\cdots,n_k)\}$ is a suitable sequence so that the series converges absolutely and uniformly in a certain region. Then
	\begin{align}\label{Th-JoshVyas-generalization-1}
		&\sum_{n_1,\cdots,n_k=0}^{\infty}c(n_1,\cdots,n_k)
		\prod_{j=1}^{k}
		\frac{(q^{\nu_j},q^{\lambda_j};q)_{n_j}}{(q^{\gamma_j},q^{\eta_j};q)_{n_j}}z_j^{n_j}\notag\\
		&\hspace{0.5cm}=
		\int_{[0,1]^k}
		\Phi(z_1t_1,\cdots,z_k t_k)
		\prod_{j=1}^{k}\mathrm{d}\mu_{\eta_j-\lambda_j,\gamma_j-\lambda_j,\gamma_j-\lambda_j+\eta_j-\nu_j,\nu_j}(t_j;q).
	\end{align}
\end{theorem}

\begin{proof}
	The result follows by the moment formula
	\[
	\int_0^1 t^\ell \mathrm{d}\mu_{\eta-\lambda,\gamma-\lambda,\gamma-\lambda+\eta-\nu,\nu}(t;q)=\frac{(q^\nu,q^\lambda;q)_\ell}{(q^\gamma,q^\eta;q)_\ell},
	\]
	which can be proved by evaluating the $q$-integral directly. This process, although a little complicated, is quite straightforward; hence we decide to omit it here.
\end{proof}
\begin{remark}
	When $k=1$, the formula \eqref{Th-JoshVyas-generalization-1} reduces to the result obtained by Joshi and Vyas. It can also be considered as Ernst's $q$-analogue for Exton's result (see \cite[p. 189, Theorem 3.1]{Ernst-2018}).
\end{remark}
Letting $k=3$ and  
\[
c(m,n,p)=(q^{\alpha_2};q)_{n+p}(q^{\beta_1};q)_{p+m}\frac{(q^{\alpha_1},q^{\eta_1};q)_m(q^{\beta_2},q^{\eta_2};q)_n(q^{\eta_3};q)_p}{(q^{\nu_1},q^{\lambda_1},q;q)_m(q^{\nu_2},q^{\lambda_2},q;q)_n(q^{\lambda_3},q^{\eta_3},q;q)_p}
\]
in Theorem \ref{Th-JoshVyas-generalization} gives the following corollary. 

\begin{corollary}\label{Cor-JoshVyas-generalization-1}
When $\min\{\Re(\gamma_i+\eta_i-\lambda_i-\nu_i), \Re(\lambda_i), \Re(\nu_i)\}>0$ $(i=1,2,3)$, we have
\begin{align}\label{Cor-JoshVyas-generalization-1-1}
	&\widetilde{\Phi}_K[\alpha_1, \alpha_2, \alpha_2, \beta_1, \beta_2, \beta_1, \gamma_1, \gamma_2, \gamma_3; q, x, y, z] \notag\\
	&\hspace{0.5cm}=\int_{[0,1]^3} \widetilde{\phi}^{(3)} \left[\begin{matrix}
			-:: -; \alpha_2; \beta_1: \alpha_1, \eta_1; \beta_2, \eta_2; \eta_3 \\
			-:: -; -; -: \nu_1, \lambda_1; \nu_2, \lambda_2; \nu_3,\lambda_3
	\end{matrix}; q, xt_1, yt_2, zt_3\right] \notag\\
	&\hspace{1.5cm}\cdot \prod_{j=1}^{3}\mathrm{d}\mu_{\eta_j-\lambda_j,\gamma_j-\lambda_j,\gamma_j-\lambda_j+\eta_j-\nu_j,\nu_j}(t_j;q),
\end{align}
where $\phi^{(3)}$ is defined by \eqref{Def-q-TripleSeries}.
\end{corollary}
\begin{remark}
	This integral \eqref{Cor-JoshVyas-generalization-1-1} is clearly \textcolor{blue}{a} $q$-analogue of the integral by Luo \emph{et al.} \cite[Theorem 1]{Luo-Xu-Raina-2022}. In particular, when $x=0$, \eqref{Cor-JoshVyas-generalization-1-1} reduces to
	 \begin{align}
	 	\widetilde{\Phi}_2[\alpha_2, \beta_2, \beta_1; \gamma_2, \gamma_3; q, y, z]
	 	&=\int_{[0,1]^2} \widetilde{\Phi}_{0:2;2}^{1:2;2} \left[\begin{matrix}
	 		\alpha_2:\beta_2, \eta_2; \beta_1, \eta_3 \\
	 		     - : \nu_2, \lambda_2; \nu_3,\lambda_3
	 	\end{matrix}; q, yt_2, zt_3\right] \notag\\
	 	&\hspace{1.5cm}\cdot \prod_{j=2}^{3}\mathrm{d}\mu_{\eta_j-\lambda_j,\gamma_j-\lambda_j,\gamma_j-\lambda_j+\eta_j-\nu_j,\nu_j}(t_j;q),
	 \end{align}
	 which is the $q$-analogue of Koschmieder's integral representation for $F_2$ (\cite[Corollary 1]{Luo-Xu-Raina-2022}). Here, $\Phi_2$ and $\Phi_{0:2;2}^{1:2;2}$ denote the $q$-Appell function \cite[p. 283, Eq. (10.2.6)]{Gasper-Rahman-2004} and the $q$-Kamp\'{e} de F\'{e}riet function \cite[p. 349]{Srivastava-Karlsson-Book-1985}, respectively.  
\end{remark}

Letting $\lambda_1=\alpha_1$, $\lambda_2=\beta_2$ and $\lambda_3=\eta_3$ in \eqref{Cor-JoshVyas-generalization-1-1}, we obtain the following corollary. 
\begin{corollary}\label{Cor-JoshVyas-generalization-2}
	When $\min\{\Re(\gamma_1+\eta_1-\alpha_1-\nu_1), \Re(\alpha_1), \Re(\nu_1)\}>0$, $\min\{\Re(\gamma_2+\eta_2-\beta_2-\nu_2), \Re(\beta_2), \Re(\nu_2)\}>0$ and $\Re(\gamma_3)>\Re(\nu_3)>0$, we have
	\begin{align}\label{Cor-JoshVyas-generalization-2-1}
		&\widetilde{\Phi}_K[\alpha_1, \alpha_2, \alpha_2, \beta_1, \beta_2, \beta_1, \gamma_1, \gamma_2, \gamma_3; q, x, y, z] \notag\\
		&\hspace{0.5cm}=\int_{[0,1]^3}
		\widetilde{\Phi}_K[\eta_1,\alpha_2,\alpha_2,\beta_1,\eta_2, \beta_1;\nu_1,\nu_2,\nu_3 
		; q, xt_1, yt_2, zt_3] \notag\\
		&\hspace{1.2cm}\cdot \mathrm{d}\mu_{\eta_1-\alpha_1,\gamma_1-\alpha_1,\gamma_1-\alpha_1+\eta_1-\nu_1,\nu_1}(t_1;q)
		\mathrm{d}\mu_{\eta_2-\beta_2,\gamma_2-\beta_2,\gamma_2-\beta_2+\eta_2-\nu_2,\nu_2}(t_2;q)
		\mathrm{d}\mu_{\nu_3,\gamma_3-\nu_3}(t_3;q).
	\end{align}
\end{corollary}
\begin{remark}
	This integral is the $q$-analogue of the integral by Luo and Raina \cite[p. 14, Theorem 4.1]{Luo-Raina-2021}. 
\end{remark}

In the next subsection, we shall show that Corollary \ref{Cor-JoshVyas-generalization-2} has a nice discrete analogue. 

\subsection{A discrete $q$-analogue of Corollary \ref{Cor-JoshVyas-generalization-2}}

Based on Gasper's formula \eqref{Gasper-DiscreteAnalog}, we obtain the following results.
\begin{theorem}\label{Th-DiscreteAnalogCor}
	\begin{align}\label{Th-DiscreteAnalogCor-1}
		&\phi^{(3)}\left[\begin{matrix}
			-:: -; q^{\alpha_2}; q^{\beta_1}: q^{\alpha_1}, q^{-r}; q^{\beta_2}, q^{-s}; q^{-t} \\
			-:: -; -; -: q^{\gamma_1}, \delta_1; q^{\gamma_2}, \delta_2; q^{\gamma_3}, \delta_3
		\end{matrix}; q, q, q, q\right]\notag\\
		&=\sum_{i=0}^r \sum_{j=0}^s \sum_{k=0}^t
		w_1(i,r;q)w_2(j,s;q)w_3(k,t;q)
		 \notag\\
		&\hspace{0.5cm}\cdot\phi^{(3)}\left[\begin{matrix}
			-:: -; q^{\alpha_2}; q^{\beta_1}: q^{\lambda_1}, q^{-i}; q^{\lambda_2}, q^{-j}; q^{-k} \\
			-:: -; -; -: q^{\mu_1}, \delta_1; q^{\mu_2}, \delta_2; q^{\mu_3}, \delta_3
		\end{matrix}; q, q, q, q\right],
	\end{align}
	where
	\begin{align*}
		w_1(i,r;q)
		&:=\frac{(q^{\alpha_1},q;q)_r}{(q^{\gamma_1},q^{\lambda_1};q)_r}
		\frac{(q^{\gamma_1+\lambda_1-\alpha_1-\mu_1};q)_{r-i}}{(q;q)_{r-i}}
		\frac{(q^{\mu_1};q)_i}{(q;q)_i}\\
		&\hspace{0.5cm}\cdot{}_3\phi_2\left[\begin{matrix}
			q^{\lambda_1-\alpha_1}, q^{\gamma_1-\alpha_1}, q^{i-r} \\
			q^{\gamma_1+\lambda_1-\alpha_1-\mu_1}, q^{1-r-\alpha_1}
		\end{matrix}; q, q^{1-i-\mu_1}\right]q^{(r-i)\mu_1},\\
		w_2(j,s;q)
		&:=\frac{(q^{\beta_2},q;q)_s}{(q^{\gamma_2},q^{\lambda_2};q)_s}\frac{(q^{\gamma_2+\lambda_2-\beta_2-\mu_2};q)_{s-j}}{(q;q)_{s-j}}\frac{(q^{\mu_2};q)_j}{(q;q)_j}\\
		&\hspace{0.5cm}\cdot{}_3\phi_2\left[\begin{matrix}
			q^{\lambda_2-\beta_2}, q^{\gamma_2-\beta_2}, q^{j-s} \\
			q^{\gamma_2+\lambda_2-\beta_2-\mu_2}, q^{1-s-\beta_2}
		\end{matrix}; q, q^{1-j-\mu_2}\right]q^{(s-j)\mu_2}
	\end{align*}
	and 
	\[
	w_3(k,t;q):=\frac{(q;q)_t}{(q^{\gamma_3};q)_t}\frac{(q^{\gamma_3-\mu_3};q)_{t-k}}{(q;q)_{t-k}}\frac{(q^{\mu_3};q)_k}{(q;q)_k}q^{(t-k)\mu_3}.
	\]
	
\end{theorem}
\begin{proof}
For convenience, we denote the right-hand side of \eqref{Th-DiscreteAnalogCor-1} by $S$. 
By \eqref{Def-q-TripleSeries}, it is not difficult to see that
\begin{align*}
	&\phi^{(3)}\left[\begin{matrix}
		-:: -; q^{\alpha_2}; q^{\beta_1}: q^{\lambda_1}, q^{-i}; q^{\lambda_2}, q^{-j}; q^{-k} \\
		-:: -; -; -: q^{\mu_1}, \delta_1; q^{\mu_2}, \delta_2; q^{\mu_3}, \delta_3
	\end{matrix}; q, q, q, q\right] \\
	&\hspace{0.5cm}=\sum_{p=0}^{k}
		\frac{(q^{\alpha_2},q^{\beta_1},q^{-k};q)_p}{(q^{\mu_3},\delta_3,q;q)_p} q^p
		{}_3\phi_2\left[\begin{matrix}
			q^{\beta_1+p}, q^{\lambda_1}, q^{-i} \\
			q^{\mu_1}, \delta_1
		\end{matrix}; q,q \right] 
		{}_3\phi_2 \left[ \begin{matrix}
			q^{\alpha_2+p}, q^{\lambda_2}, q^{-j} \\
			q^{\mu_2}, \delta_2
		\end{matrix};q,q \right].
\end{align*}
Then
\begin{align}
	S&=\sum_{i=0}^r \sum_{j=0}^s \sum_{k=0}^t
	w_1(i,r;q)w_2(j,s;q)w_3(k,t;q)
		\sum_{p=0}^{k}
		\frac{(q^{\alpha_2},q^{\beta_1},q^{-k};q)_p}{(q^{\mu_3},\delta_3,q;q)_p} q^p\notag\\
		&\hspace{0.5cm}\cdot {}_3\phi_2\left[\begin{matrix}
			q^{\beta_1+p}, q^{\lambda_1}, q^{-i} \\
			q^{\mu_1}, \delta_1
		\end{matrix}; q,q \right] 
		{}_3\phi_2 \left[ \begin{matrix}
			q^{\alpha_2+p}, q^{\lambda_2}, q^{-j} \\
			q^{\mu_2}, \delta_2
		\end{matrix};q,q \right]\notag\\
	&=\sum_{k=0}^t w_3(k,t;q)
	\sum_{p=0}^{k}
	\frac{(q^{\alpha_2},q^{\beta_1},q^{-k};q)_p}{(q^{\mu_3},\delta_3,q;q)_p} q^p\notag\\
	&\hspace{0.5cm}
	\sum_{i=0}^r 
	w_1(i,r;q){}_3\phi_2\left[\begin{matrix}
		q^{\beta_1+p}, q^{\lambda_1}, q^{-i} \\
		q^{\mu_1}, \delta_1
	\end{matrix}; q,q \right] \notag\\
	&\hspace{0.5cm}\sum_{j=0}^s
	w_2(j,s;q)
	{}_3\phi_2 \left[ \begin{matrix}
		q^{\alpha_2+p}, q^{\lambda_2}, q^{-j} \\
		q^{\mu_2}, \delta_2
	\end{matrix};q,q \right].\label{Th-DiscreteAnalogCor-Proof-1}
\end{align}

Letting $\alpha\rightarrow q^{\alpha_1}$, $\beta\rightarrow q^{\beta_1+p}$, $\gamma\rightarrow q^{\gamma_1}$, $\delta\rightarrow\delta_1$, $\mu\rightarrow q^{\lambda_1}$, $\lambda\rightarrow q^{\alpha_1}$ and $\nu\rightarrow q^{\mu_1}$ in \eqref{Gasper-DiscreteAnalog}, we have
\begin{align}\label{Th-DiscreteAnalogCor-Proof-3}
\sum_{i=0}^r 
w_1(i,r;q){}_3\phi_2\left[\begin{matrix}
	q^{\beta_1+p}, q^{\lambda_1}, q^{-i} \\
	q^{\mu_1}, \delta_1
\end{matrix}; q,q \right]
&={}_3\phi_2 \left[ \begin{matrix}
	q^{\alpha_1}, q^{\beta_1+p}, q^{-r} \\
	q^{\gamma_1}, \delta_1
\end{matrix};q,q \right]\notag\\
&=\sum_{i=0}^{r}\frac{(q^{\alpha_1},q^{\beta_1},q^{-r};q)_i}{(q^{\gamma_1},\delta_1,q;q)_i}\frac{(q^{\beta_1+i};q)_p}{(q^{\beta_1};q)_p}q^i.
\end{align}
Letting $\alpha\rightarrow q^{\alpha_2+p}$, $\beta\rightarrow q^{\beta_2}$, $\gamma\rightarrow q^{\gamma_2}$, $\delta\rightarrow\delta_2$, $\mu\rightarrow q^{\lambda_2}$, $\lambda\rightarrow q^{\beta_2}$ and $\nu\rightarrow q^{\mu_2}$ in \eqref{Gasper-DiscreteAnalog}, we obtain
\begin{align}\label{Th-DiscreteAnalogCor-Proof-4}
\sum_{j=0}^s
w_2(j,s;q)
{}_3\phi_2 \left[ \begin{matrix}
	q^{\alpha_2+p}, q^{\lambda_2}, q^{-j} \\
	q^{\mu_2}, \delta_2
\end{matrix};q,q \right]
&={}_3\phi_2 \left[ \begin{matrix}
	q^{\alpha_2+p}, q^{\beta_2}, q^{-s} \\
	q^{\gamma_2}, \delta_2
\end{matrix};q,q \right]\notag\\
&=\sum_{j=0}^{s}\frac{(q^{\alpha_2},q^{\beta_2},q^{-s};q)_j}{(q^{\gamma_2},\delta_2,q;q)_j}\frac{(q^{\alpha_2+j};q)_p}{(q^{\alpha_2};q)_p}q^j.
\end{align}

Now, substituting \eqref{Th-DiscreteAnalogCor-Proof-3} and \eqref{Th-DiscreteAnalogCor-Proof-4} into \eqref{Th-DiscreteAnalogCor-Proof-1} gives
\begin{align}\label{Th-DiscreteAnalogCor-Proof-5}
	S&=\sum_{i=0}^{r}\frac{(q^{\alpha_1},q^{\beta_1},q^{-r};q)_i}{(q^{\gamma_1},\delta_1,q;q)_i}q^i
	\sum_{j=0}^{s}\frac{(q^{\alpha_2},q^{\beta_2},q^{-s};q)_j}{(q^{\gamma_2},\delta_2,q;q)_j}q^j\notag\\
	&\hspace{0.5cm}\sum_{k=0}^t w_3(k,t;q)
	\sum_{p=0}^{k}
	\frac{(q^{\alpha_2+j},q^{\beta_1+i},q^{-k};q)_p}{(q^{\mu_3},\delta_3,q;q)_p} q^p.
\end{align}
Letting $\alpha\rightarrow q^{\alpha_2+j}$, $\beta\rightarrow q^{\beta_1+i}$, $\gamma\rightarrow q^{\gamma_3}$, $\delta\rightarrow\delta_3$, $\mu=\lambda$ and $\nu\rightarrow q^{\mu_3}$ in \eqref{Gasper-DiscreteAnalog} gives
\begin{equation}\label{Th-DiscreteAnalogCor-Proof-6}
\sum_{k=0}^t w_3(k,t;q)
\sum_{p=0}^{k}
\frac{(q^{\alpha_2+j},q^{\beta_1+i},q^{-k};q)_p}{(q^{\mu_3},\delta_3,q;q)_p} q^p
={}_{3}\phi_{2}\left[\begin{matrix}
	q^{\alpha_2+j}, q^{\beta_1+i}, q^{-n}\\
	q^{\gamma_3},\delta_3 
\end{matrix};q,q\right].
\end{equation}
Substituting \eqref{Th-DiscreteAnalogCor-Proof-6} into \eqref{Th-DiscreteAnalogCor-Proof-5}, we obtain
\begin{align*}
	S&=\sum_{i=0}^{r}\frac{(q^{\alpha_1},q^{\beta_1},q^{-r};q)_i}{(q^{\gamma_1},\delta_1,q;q)_i}q^i
	\sum_{j=0}^{s}\frac{(q^{\alpha_2},q^{\beta_2},q^{-s};q)_j}{(q^{\gamma_2},\delta_2,q;q)_j}q^j\sum_{k=0}^t \frac{(q^{\alpha_2+j},q^{\beta_1+i},q^{-t};q)_k}{(q^{\gamma_3},\delta_3,q;q)_k}q^k\\
	&=\sum_{i=0}^{r}\sum_{j=0}^{s}\sum_{k=0}^{t}
	(q^{\alpha_2};q)_{j+k}
	(q^{\beta_1};q)_{k+i}
	\frac{(q^{\alpha_1},q^{-r};q)_i}{(q^{\gamma_1},\delta_1,q;q)_i}
	\frac{(q^{\beta_2},q^{-s};q)_j}{(q^{\gamma_2},\delta_2,q;q)_j} \frac{(q^{-t};q)_k}{(q^{\gamma_3},\delta_3,q;q)_k}q^{i+j+k}
\end{align*}
Interpreting the above triple sum as the $\phi^{(3)}$-function leads us to the formula \eqref{Th-DiscreteAnalogCor-1}. 
\end{proof}

By letting $i\rightarrow r-i$, $j\rightarrow s-j$, $k\rightarrow t-k$, $\delta_1 \rightarrow q^{1-r}/x$, $\delta_2 \rightarrow q^{1-s}/y$, $\delta_3 \rightarrow q^{1-t}/z$ in \eqref{Th-DiscreteAnalogCor-1}, we have
\begin{align}\label{Th-DiscreteAnalogCor-5}
	&\phi^{(3)}\left[\begin{matrix}
		-:: -; q^{\alpha_2}; q^{\beta_1}: q^{\alpha_1}, q^{-r}; q^{\beta_2}, q^{-s}; q^{-t} \\
		-:: -; -; -: q^{\gamma_1}, q^{1-r}/x; q^{\gamma_2}, q^{1-s}/y; q^{\gamma_3},  q^{1-t}/z
	\end{matrix}; q, q, q, q\right]\notag\\
	&=\sum_{i=0}^r \sum_{j=0}^s \sum_{k=0}^t
	w_1(r-i,r;q)w_2(s-j,s;q)w_3(t-k,t;q)
	\notag\\
	&\hspace{0.5cm}\cdot\phi^{(3)}\left[\begin{matrix}
		-:: -; q^{\alpha_2}; q^{\beta_1}: q^{\lambda_1}, q^{-i}; q^{\lambda_2}, q^{-j}; q^{-k} \\
		-:: -; -; -: q^{\mu_1}, q^{1-r}/x; q^{\mu_2}, q^{1-s}/y; q^{\mu_3},  q^{1-t}/z
	\end{matrix}; q, q, q, q\right].
\end{align}
Note that
\begin{align}
	\lim_{r\rightarrow+\infty}w_1(r-i,r;q)
	&=\frac{(q^{\alpha_1},q^{\mu_1};q)_\infty}{(q^{\gamma_1},q^{\lambda_1};q)_\infty}\frac{(q^{\gamma_1+\lambda_1-\alpha_1-\mu_1};q)_i}{(q;q)_i}q^{i\mu_1}\notag\\
	&\cdot {}_3\phi_1\left[\begin{matrix}
		q^{\lambda_1-\alpha_1}, q^{\gamma_1-\alpha_1}, q^{-i} \\
		q^{\gamma_1+\lambda_1-\alpha_1-\mu_1}
	\end{matrix}; q, q^{\alpha_1-\mu_1+i} \right], \label{Th-DiscreteAnalogCor-2}\\
	\lim_{s\rightarrow+\infty}w_2(s-j,s;q)
	&=\frac{(q^{\beta_2},q^{\mu_2};q)_\infty}{(q^{\gamma_2},q^{\lambda_2};q)_\infty}\frac{(q^{\gamma_2+\lambda_2-\beta_2-\mu_2};q)_j}{(q;q)_j} q^{j\mu_3}\notag\\
	&\cdot {}_3\phi_1\left[\begin{matrix}
		q^{\lambda_2-\beta_2}, q^{\gamma_2-\beta_2}, q^{-j} \\
		q^{\gamma_2+\lambda_2-\beta_2-\mu_2}
	\end{matrix}; q, q^{\beta_2-\mu_2+j}\right]
	\label{Th-DiscreteAnalogCor-3}
\end{align}
and 
\begin{equation}\label{Th-DiscreteAnalogCor-4}
	\lim_{t\rightarrow+\infty}w_3(t-k,t;q):=\frac{(q^{\mu_3};q)_\infty}{(q^{\gamma_3};q)_\infty}\frac{(q^{\gamma_3-\mu_3};q)_k}{(q;q)_k}q^{k\mu_3}.
\end{equation}
Letting $r,s,t\rightarrow+\infty$ in \eqref{Th-DiscreteAnalogCor-5} and making use of \eqref{Th-DiscreteAnalogCor-2}, \eqref{Th-DiscreteAnalogCor-3} and \eqref{Th-DiscreteAnalogCor-4}, we obtain
\begin{align*}
	&\widetilde{\Phi}_K[\alpha_1,\alpha_2,\alpha_2,\beta_1,\beta_2,\beta_1;\gamma_1,\gamma_2,\gamma_3; q, x, y, z]\notag\\
	&=\frac{(q^{\alpha_1},q^{\mu_1},q^{\gamma_1-\lambda_1-\alpha_1-\mu_1};q)_\infty}{(q^{\gamma_1},q^{\lambda_1},q;q)_\infty}
	\frac{(q^{\beta_2},q^{\mu_2},q^{\gamma_2+\lambda_2-\beta_2-\mu_2};q)_\infty}{(q^{\gamma_2},q^{\lambda_2},q;q)_\infty}
	\frac{(q^{\mu_3},q^{\gamma_3-\mu_3};q)_\infty}{(q^{\gamma_3},q;q)_\infty}
	\notag\\
	&\hspace{0.5cm}\cdot\sum_{i,j,k=0}^{\infty}
	\frac{(q^{1+i};q)_\infty}{(q^{\gamma_1+\lambda_1-\alpha_1-\mu_1+i};q)_\infty}
	\frac{(q^{1+j};q)_\infty}{(q^{\gamma_2+\lambda_2-\beta_2-\mu_2+j};q)_\infty} 
	\frac{(q^{1+k};q)_\infty}{(q^{\gamma_3-\mu_3+k};q)_\infty} 
	q^{i\mu_1+j\mu_2+k\mu_3} \notag \\
	&\hspace{0.5cm}\cdot {}_3\phi_1\left[\begin{matrix}
		q^{\lambda_1-\alpha_1}, q^{\gamma_1-\alpha_1}, q^{-i} \\
		q^{\gamma_1+\lambda_1-\alpha_1-\mu_1}
	\end{matrix}; q, q^{\alpha_1-\mu_1+i} \right]
	{}_3\phi_1\left[\begin{matrix}
		q^{\lambda_2-\beta_2}, q^{\gamma_2-\beta_2}, q^{-j} \\
		q^{\gamma_2+\lambda_2-\beta_2-\mu_2}
	\end{matrix}; q, q^{\beta_2-\mu_2+j}\right] \notag\\
	&\hspace{0.5cm}\cdot
	\widetilde{\Phi}_K[\lambda_1,\alpha_2,\alpha_2,\beta_1,\lambda_2,\beta_1;\mu_1,\mu_2,\mu_3; q, xq^i, yq^j, zq^k].
\end{align*}
Using definition \eqref{Def-q-integral} of the multiple $q$-integral, the expression can be further reduced to \eqref{Cor-JoshVyas-generalization-2-1}.

\subsection{A $q$-analogue of Theorem \ref{Th-ErdelyiIntegral-FK}}

Let $f:\mathbb{C}\rightarrow\mathbb{C}$ be a function of one variable. We define the $q$-shift operator $\mathbb{E}_{q}$ by $\mathbb{E}_{q}f(x)=f(qx)$ $(0<q<1)$ \cite[Definition 1.3]{Liu-2023}. Inductively, for any $n\in\mathbb{Z}_{\geq1}$, 
\[
\mathbb{E}_{q}^n f(x)=f(q^n x)
\] 
and in particular, $\mathbb{E}_{q}^0 f(x)=f(x)$. In addition, $\displaystyle \lim_{q\rightarrow1^{-}} \mathbb{E}_{q}^{n}f(x)=f(x)$ . For $k$-dimensional case, we use the notation $\mathbb{E}_{q,x_i}$ to denote the operator acting on the $i$-th variable of $f$, namely, 
\[
\mathbb{E}_{q,x_i} f(x_1,\cdots,x_i,\cdots, x_k)
=f(x_1,\cdots,q x_i,\cdots, x_k), ~ i=1,\cdots, k.
\]
We shall frequently use the formula
\begin{equation}\label{OF-qShift}
\mathbb{E}_{q,x}^n \mathbb{E}_{q,y}^m f(x,y)=f(q^n x, q^m y), ~~~(n,m)\in\mathbb{Z}_{\geq0}^2.
\end{equation}

If $\displaystyle \mathcal{E}(x)=\sum_{n=0}^{\infty}a_n x^n$ $(|x|<R)$ is an analytic functions, then 
\[
\mathcal{E}(t~\mathbb{E}_{q,x})f(x)=\sum_{n=0}^{\infty}a_n t^n f(q^n x), ~~~|t|<R,
\]
where $f(x)$ is also assumed to be bounded (or, in many situations, analytic) in a neighborhood of $0$. We have 
\[
\lim_{q\rightarrow 1^{-}}
\mathcal{E}(t~\mathbb{E}_{q,x})f(x)=\mathcal{E}(t)f(x).
\]
The limit helps overcome the difficulty in deriving the $q$-analogue of Theorem \ref{Th-ErdelyiIntegral-FK}.

\begin{theorem}\label{Th-qErdelyiIntegral-qFK}
	Let $\Re(\alpha_1+\eta_1)>\Re(\lambda_1)>0$, $\Re(\beta_2+\mu_2)>\Re(\lambda_2)>0$ and $\Re(\gamma_3)>\Re(\beta_1)>0$. Then we have
	\begin{align}\label{Th-qErdelyiIntegral-qFK-1}
		&\widetilde{\Phi}_K[\alpha_1, \alpha_2, \alpha_2, \beta_1, \beta_2, \beta_1; \alpha_1+\eta_1, \beta_2+\mu_2, \gamma_3; q, x, y, z]
		=\int_{[0,1]^3}
		\frac{(uxq^{\lambda_3},vyq^{\eta_2};q)_{\infty}}{(ux,vy;q)_{\infty}}\notag\\
		&\hspace{0.5cm}\cdot \phi^{(3)}\Bigg[\begin{matrix}
				- :: -; q^{\eta_2}; q^{\lambda_3}: q^{\lambda_1-\eta_1}, u^{-1}; q^{\lambda_2-\mu_2}, v^{-1};  -\\
				- :: -; -; -: q^{\lambda_1}, q/(ux); q^{\lambda_2}, q/(vy); uxq^{\lambda_3}, vyq^{\eta_2}, q^{\lambda_3}
			\end{matrix};\notag\\
		&\hspace{1.5cm}q, q, q, wzq^{\alpha_2-\eta_2} \mathbb{E}_{q,x} \mathbb{E}_{q,y} \Bigg]\notag\\
		&\hspace{0.5cm}\cdot \widetilde{\Phi}_K\big[\alpha_1, \alpha_2-\eta_2, \alpha_2-\eta_2, \beta_1-\lambda_3, \beta_2, \beta_1-\lambda_3; \alpha_1-\lambda_1+\eta_1, \beta_2-\lambda_2+\mu_2, \beta_1-\lambda_3;\notag\\
		&\hspace{1.5cm}q, uxq^{\lambda_3}, vyq^{\eta_2}, wz\big]
		\mathrm{d}\mu_{\alpha_1-\lambda_1+\eta_1, \lambda_1}(u;q) \mathrm{d}\mu_{\beta_2-\lambda_2+\mu_2, \lambda_2}(v;q) \mathrm{d}\mu_{\beta_1, \gamma_3-\beta_1}(w;q).
	\end{align}
\end{theorem}
\begin{proof}
We use the method from Section \ref{Section-3}. Let $I$ denote the triple $q$-integral in \eqref{Th-qErdelyiIntegral-qFK-1}. By the definition \eqref{Def-q-TripleSeries} of the $\phi^{(3)}$-function, we have 
\begin{align*}
	&\phi^{(3)}\left[\begin{matrix}
			- :: -; q^{\eta_2}; q^{\lambda_3}: q^{\lambda_1-\eta_1}, u^{-1}; q^{\lambda_2-\mu_2}, v^{-1};  -\\
			- :: -; -; -: q^{\lambda_1}, q/(ux); q^{\lambda_2}, q/(vy); uxq^{\lambda_3}, vyq^{\eta_2}, q^{\lambda_3}
	\end{matrix}; q, q, q, wzq^{\alpha_2-\eta_2} \mathbb{E}_{q,x} \mathbb{E}_{q,y} \right]\\
	&\hspace{0.5cm}=\sum_{k=0}^{\infty} \frac{(q^{\eta_2};q)_k}{(uxq^{\lambda_3}, vyq^{\eta_2};q)_k} {}_3\phi_2 \left[ \begin{matrix}
			q^{\lambda_3+k}, q^{\lambda_1-\eta_1}, u^{-1} \\
			q^{\lambda_1}, q/(ux)
		\end{matrix}; q , q \right]\\
	&\hspace{1.5cm}\cdot {}_3\phi_2 \left[\begin{matrix}
			q^{\eta_2+k}, q^{\lambda_2-\mu_2}, v^{-1} \\
			q^{\lambda_2}, q/(vy)
		\end{matrix}; q, q \right] \frac{(wzq^{\alpha_2-\eta_2} \mathbb{E}_{q,x} \mathbb{E}_{q,y})^k}{(q;q)_k}.
\end{align*}
Then, in view of \eqref{OF-qShift} and \eqref{Def-qFK-2varphi1}, we have
\begin{align}
	I&=\int_{[0,1]^3}
	\frac{(uxq^{\lambda_3},vyq^{\eta_2};q)_{\infty}}{(ux,vy;q)_{\infty}} \sum_{k=0}^{\infty} \frac{(q^{\eta_2};q)_k}{(uxq^{\lambda_3}, vyq^{\eta_2};q)_k} \notag\\
	&\hspace{0.5cm}\cdot
	{}_3\phi_2 \left[ \begin{matrix}
		q^{\lambda_3+k}, q^{\lambda_1-\eta_1}, u^{-1} \\
		q^{\lambda_1}, q/(ux)
	\end{matrix}; q , q \right] 
	{}_3\phi_2 \left[\begin{matrix}
		q^{\eta_2+k}, q^{\lambda_2-\mu_2}, v^{-1} \\
		q^{\lambda_2}, q/(vy)
	\end{matrix}; q, q\right] \frac{(wzq^{\alpha_2-\eta_2})^k}{(q;q)_k} \notag\\
	&\hspace{0.5cm}\cdot \widetilde{\Phi}_K\big[\alpha_1, \alpha_2-\eta_2, \alpha_2-\eta_2, \beta_1-\lambda_3, \beta_2, \beta_1-\lambda_3; \alpha_1-\lambda_1+\eta_1, \beta_2-\lambda_2+\mu_2, \beta_1-\lambda_3;\notag\\
	&\hspace{1cm} uxq^{k+\lambda_3}, vyq^{k+\eta_2}, wz\big]\mathrm{d}\mu_{\alpha_1-\lambda_1+\eta_1, \lambda_1}(u;q)\mathrm{d}\mu_{\beta_2-\lambda_2+\mu_2, \lambda_2}(v;q)\mathrm{d}\mu_{\beta_1, \gamma_3-\beta_1}(w;q)\notag\\
	&=\int_{[0,1]^3}
	 \sum_{k=0}^{\infty} \frac{(q^{\eta_2};q)_k}{(q;q)_k}\frac{(uxq^{k+\lambda_3},vyq^{k+\eta_2};q)_{\infty}}{(ux,vy;q)_{\infty}}(wzq^{\alpha_2-\eta_2})^k \notag\\
	&\hspace{0.5cm}\cdot
	{}_3\phi_2 \left[ \begin{matrix}
		q^{\lambda_3+k}, q^{\lambda_1-\eta_1}, u^{-1} \\
		q^{\lambda_1}, q/(ux)
	\end{matrix}; q,q \right] 
	{}_3\phi_2 \left[\begin{matrix}
		q^{\eta_2+k}, q^{\lambda_2-\mu_2}, v^{-1} \\
		q^{\lambda_2}, q/(vy)
	\end{matrix}; q,q\right] \notag\\
	&\hspace{0.5cm}\cdot \sum_{\ell=0}^{\infty} \frac{(q^{\alpha_2-\eta_2};q)_\ell}{(q;q)_\ell} {}_2\widetilde{\phi}_1 \left[\begin{matrix}
		\beta_1-\lambda_3+\ell, \alpha_1\\
		\alpha_1-\lambda_1+\eta_1
	\end{matrix}; q, uxq^{k+\lambda_3}
	\right] {}_2\widetilde{\phi}_1 \left[\begin{matrix}
		\alpha_2-\eta_2+\ell, \beta_2 \\
		\beta_2-\lambda_2+\mu_2
	\end{matrix}; q, vyq^{k+\eta_2}
	\right]\notag\\
	&\hspace{0.5cm}\cdot (wz)^{\ell}\mathrm{d}\mu_{\alpha_1-\lambda_1+\eta_1, \lambda_1}(u;q) \mathrm{d}\mu_{\beta_2-\lambda_2+\mu_2, \lambda_2}(v;q) \mathrm{d}\mu_{\beta_1, \gamma_3-\beta_1}(w;q)\notag\\
	&=\sum_{k=0}^{\infty} \frac{(q^{\eta_2};q)_k}{(q;q)_k} (zq^{\alpha_2-\eta_2})^k \sum_{\ell=0}^{\infty} \frac{(q^{\alpha_2-\eta_2};q)_\ell}{(q;q)_\ell} z^\ell \int_0^1 w^{k+\ell} \mathrm{d}\mu_{\beta_1,\gamma_3-\beta_1}(w;q) \notag\\
	&\hspace{0.5cm}\cdot\int_0^1 \frac{(uxq^{\lambda_3+k};q)_{\infty}}{(ux;q)_{\infty}} {}_2\widetilde{\phi}_1 \left[ \begin{matrix}
		\beta_1-\lambda_3+\ell, \alpha_1 \\
		\alpha_1-\lambda_1+\eta_1
	\end{matrix}; q, uxq^{\lambda_3+k}
	\right] \notag\\
	&\hspace{2cm}\cdot{}_3\phi_2 \left[ \begin{matrix}
		q^{\lambda_3+k}, q^{\lambda_1-\eta_1}, u^{-1} \\
		q^{\lambda_1}, q/(ux)
	\end{matrix}; q , q \right]\mathrm{d} \mu_{\alpha_1-\lambda_1+\eta_1,\lambda_1}(u;q) \notag\\
	&\hspace{0.5cm}\cdot\int_0^1 \frac{(vyq^{\eta_2+k};q)_{\infty}}{(vy;q)_{\infty}} 
	{}_2\widetilde{\phi}_1 \left[ \begin{matrix}
		\alpha_2-\eta_2+\ell, \beta_2 \\
		\beta_2-\lambda_2+\mu_2
	\end{matrix}; q, vyq^{\eta_2+k}
	\right]\notag\\
	&\hspace{2cm}\cdot{}_3\phi_2 \left[ \begin{matrix}
		q^{\eta_2+k}, q^{\lambda_2-\mu_2}, v^{-1} \\
		q^{\lambda_2}, q/(vy)
	\end{matrix}; q, q \right] \mathrm{d} \mu_{\beta_2-\lambda_2+\mu_2,\lambda_2}(v;q).
	\label{Th-qErdelyiIntegral-qFK-Proof-1}
\end{align}

By letting $\alpha\rightarrow\beta_1+k+\ell$, $\beta\rightarrow\alpha_1$, $\gamma\rightarrow\alpha_1+\eta_1$, $\lambda\rightarrow\alpha_1-\lambda_1+\eta_1$ and $\alpha'\rightarrow\lambda_3+k$ in \eqref{eq:2phi1 Erdelyi1}, we have
\begin{align}\label{Th-qErdelyiIntegral-qFK-Proof-2}
	&\int_0^1 \frac{(uxq^{\lambda_3+k};q)_{\infty}}{(ux;q)_{\infty}} {}_2\widetilde{\phi}_1 \left[\begin{matrix}
			\beta_1-\lambda_3+\ell, \alpha_1 \\
			\alpha_1-\lambda_1+\eta_1
		\end{matrix}; q, uxq^{\lambda_3+k}
		\right] \notag\\
	&\hspace{0.5cm}\cdot{}_3\phi_2 \left[ \begin{matrix}
			q^{\lambda_3+k}, q^{\lambda_1-\eta_1}, u^{-1} \\
			q^{\lambda_1}, q/(ux)
		\end{matrix}; q, q \right]
		\mathrm{d}\mu_{\alpha_1-\lambda_1+\eta_1,\lambda_1}(u;q)
		={}_2\widetilde{\phi}_1\left[\begin{matrix}
			\beta_1+k+\ell, \alpha_1\\
			\alpha_1+\eta_1
		\end{matrix}; q, x\right],
\end{align}
where $\Re(\alpha_1+\eta_1)>\Re(\lambda_1)>0$. Similarly, let $\alpha\rightarrow\alpha_2+k+\ell$, $\beta\rightarrow\beta_2$, $\gamma\rightarrow\beta_2+\mu_2$, $\lambda=\beta_2-\lambda_2+\mu_2$ and $\alpha'\rightarrow\eta_2+k$ in \eqref{eq:2phi1 Erdelyi1}, then
\begin{align}\label{Th-qErdelyiIntegral-qFK-Proof-3}
	&\int_0^1 \frac{(vyq^{\eta_2+k};q)_{\infty}}{(vy;q)_{\infty}} {}_2\widetilde{\phi}_1\left[\begin{matrix}
			\alpha_2-\eta_2+\ell, \beta_2 \\
			\beta_2-\lambda_2+\mu_2
		\end{matrix}; q, vyq^{\eta_2+k}
		\right]\notag\\
	&\hspace{0.5cm}\cdot{}_3\phi_2 \left[ \begin{matrix}
			q^{\eta_2+k}, q^{\lambda_2-\mu_2}, v^{-1} \\
			q^{\lambda_2}, q/(vy)
		\end{matrix}; q, q\right]\mathrm{d} \mu_{\beta_2-\lambda_2+\mu_2,\lambda_2}(v;q)
	={}_2\widetilde{\phi}_1 \left[\begin{matrix}
			\alpha_2+k+\ell, \beta_2 \\
			\beta_2+\mu_2
		\end{matrix}; q, y\right],
\end{align}
where $\Re(\beta_2+\mu_2)>\Re(\lambda_2)>0$. Substituting \eqref{Th-qErdelyiIntegral-qFK-Proof-2} and \eqref{Th-qErdelyiIntegral-qFK-Proof-3} in \eqref{Th-qErdelyiIntegral-qFK-Proof-1} gives
\begin{align}\label{Th-qErdelyiIntegral-qFK-Proof-4}
	I&=\sum_{k=0}^{\infty} \frac{(q^{\eta_2};q)_k}{(q;q)_k} (q^{\alpha_2-\eta_2})^{k} 
	\sum_{\ell=0}^{\infty} \frac{(q^{\alpha_2-\eta_2};q)_\ell}{(q;q)_\ell} \frac{(q^{\beta_1};q)_{k+\ell}}{(q^{\gamma_3};q)_{k+\ell}} z^{k+\ell}\notag\\
	&\hspace{1cm}\cdot {}_2\widetilde{\phi}_1 \left[\begin{matrix}
		\beta_1+k+\ell, \alpha_1 \\
		\alpha_1+\eta_1
	\end{matrix}; q, x\right]
	{}_2\widetilde{\phi}_1 \left[\begin{matrix}
			\alpha_2+k+\ell, \beta_2 \\
			\beta_2+\mu_2
		\end{matrix}; q, y\right]\notag\\
	&=\sum_{k=0}^{\infty} \frac{(q^{\beta_1};q)_{k}}{(q^{\gamma_3};q)_{k}} z^{k} {}_2\widetilde{\phi}_1 \left[\begin{matrix}
		\beta_1+k, \alpha_1 \\
		\alpha_1+\eta_1
	\end{matrix}; q, x \right] {}_2\widetilde{\phi}_1 \left[\begin{matrix}
		\alpha_2+k, \beta_2 \\
		\beta_2+\mu_2
	\end{matrix}; q,y\right]\notag\\
	&\hspace{1cm}\cdot\sum_{\ell=0}^{k} \frac{(q^{\eta_2};q)_{k-\ell}}{(q;q)_{k-\ell}}  \frac{(q^{\alpha_2-\eta_2};q)_\ell}{(q;q)_\ell} (q^{\alpha_2-\eta_2})^{k-\ell}\notag\\
	&=\sum_{k=0}^{\infty} \frac{(q^{\alpha_2},q^{\beta_1};q)_{k}}{(q^{\gamma_3},q;q)_{k}} z^{k} {}_2\widetilde{\phi}_1 \left[ \begin{matrix}
		\beta_1+k, \alpha_1 \\
		\alpha_1+\eta_1
	\end{matrix}; q, x \right] 
	{}_2\widetilde{\phi}_1 \left[\begin{matrix}
		\alpha_2+k, \beta_2\\
		\beta_2+\mu_2
	\end{matrix}; q, y \right].
\end{align}
By interpreting the last series in \eqref{Th-qErdelyiIntegral-qFK-Proof-4} as the $q$-Saran function $\widetilde{\Phi}_K$, we obtain the desired result \eqref{Th-qErdelyiIntegral-qFK-1}. 
\end{proof}
\begin{remark}
The formula \eqref{Th-qErdelyiIntegral-qFK-1} obtained here involves the $q$-shift operators $\mathbb{E}_{q,x}$ and $\mathbb{E}_{q,y}$, which makes it appears slightly different from the familiar integrals associated with ${}_{2}\phi_{1}$. It would therefore be interesting to explore whether alternative $q$-Erd\'{e}lyi-type integrals for $\Phi_K$ can be derived through a different approach. We conjecture that the use of \emph{$q$-fractional integration by parts} \cite[p. 159, Theorem 5.14]{Annaby-Mansour-2012} may offer a viable path for addressing this problem. 
\end{remark}

Finally, we show that in some cases, the integrand in \eqref{Th-qErdelyiIntegral-qFK} can be simplified so that the operators $\mathbb{E}_{q,x}$ and $\mathbb{E}_{q,y}$ do not appear any more.
\begin{corollary}\label{Cor-qErdelyiIntegral-qFK}
	When $\min\{\Re(\alpha_1),\Re(\eta_1),\Re(\beta_2),\Re(\mu_2)\}>0$ and $\Re(\gamma_3)>\Re(\beta_1)>0$, we have 
\begin{align}\label{Cor-qErdelyiIntegral-qFK-1}
	&\widetilde{\Phi}_K [\alpha_1,\alpha_2,\alpha_2,\beta_1,\beta_2,\beta_1;\alpha_1+\eta_1,\beta_2+\mu_2,\gamma_3;q,x,y,z]=\int_{[0,1]^3} \frac{(uxq^{\beta_1},vyq^{\alpha_2};q)_{\infty}}{(ux,vy;q)_{\infty}} \notag\\
	&\hspace{1cm}\cdot {}_3\phi_2\left[\begin{matrix}
				q^{\alpha_2},0,0 \\
				uxq^{\beta_1},vyq^{\alpha_2}
	\end{matrix};q,wz\right]\mathrm{d}\mu_{\alpha_1,\eta_1}(u;q) \mathrm{d}\mu_{\beta_2,\mu_2}(v;q) \mathrm{d}\mu_{\beta_1,\gamma_3-\beta_1}(w;q) .
	\end{align}
\end{corollary}
\begin{proof}
	By letting $\lambda_1\rightarrow \eta_1$ and $\lambda_2\rightarrow \mu_2$ in \eqref{Th-qErdelyiIntegral-qFK-1}, we obtain
\begin{align}\label{Cor-qErdelyiIntegral-qFK-Proof-1}
	&\widetilde{\Phi}_K[\alpha_1, \alpha_2, \alpha_2, \beta_1, \beta_2, \beta_1; \alpha_1+\eta_1, \beta_2+\mu_2, \gamma_3; q, x, y, z]
	=\int_{[0,1]^3}
	\frac{(uxq^{\lambda_3},vyq^{\eta_2};q)_{\infty}}{(ux,vy;q)_{\infty}}\notag\\
	&\hspace{0.5cm}\cdot \mathcal{K}(q;x,y,z;u,v,w)\mathrm{d}\mu_{\alpha_1, \eta_1}(u;q) \mathrm{d}\mu_{\beta_2, \mu_2}(v;q) \mathrm{d}\mu_{\beta_1, \gamma_3-\beta_1}(w;q),
\end{align}
where
\begin{align*}
\mathcal{K}(q;x,y,z;u,v,w)&:=
{}_{3}\phi_{2}\left[\begin{matrix}
	q^{\eta_2}, 0, 0\\
	uxq^{\lambda_3},vyq^{\eta_2}
\end{matrix};q,wzq^{\alpha_2-\eta_2}\mathbb{E}_{q,x}\mathbb{E}_{q,y}\right]\notag\\
&\hspace{1cm}\frac{(uxq^{\beta_1},vyq^{\alpha_2};q)_{\infty}}{(uxq^{\lambda_3},vyq^{\eta_2};q)_{\infty}} 
{}_{3}\phi_{2}\left[\begin{matrix}
	q^{\alpha_2-\eta_2}, 0, 0\\
	uxq^{\beta_1},vyq^{\alpha_2}
\end{matrix};q,wz\right].
\end{align*}
To evaluate $\mathcal{K}(q;x,y,z;u,v,w)$, we first note that
\begin{align*}
\mathcal{K}(q;x,y,z;u,v,w)
&=\sum_{k=0}^{\infty} \frac{(q^{\eta_2};q)_k(wzq^{\alpha_2-\eta_2})^k}{(uxq^{\lambda_3},vyq^{\eta_2},q;q)_k}  \frac{(uxq^{\beta_1+k},vyq^{\alpha_2+k};q)_{\infty}}{(uxq^{\lambda_3+k},vyq^{\eta_2+k};q)_{\infty}} \\
&\hspace{0.5cm}\cdot \sum_{p=0}^{\infty} \frac{(q^{\alpha_2-\eta_2};q)_p (wz)^p	}{(uxq^{\beta_1+k},vyq^{\alpha_2+k},q;q)_p}\\
&=\frac{(uxq^{\beta_1},vyq^{\alpha_2};q)_{\infty}}{(uxq^{\lambda_3},vyq^{\eta_2};q)_{\infty}} \sum_{k=0}^{\infty}\sum_{p=0}^{\infty} 
\frac{(q^{\eta_2};q)_k(q^{\alpha_2-\eta_2};q)_pq^{k(\alpha_2-\eta_2)}
(wz)^{k+p}}{(uxq^{\beta_1},vyq^{\alpha_2};q)_{k+p}(q;q)_k(q;q)_p}\\
&=\frac{(uxq^{\beta_1},vyq^{\alpha_2};q)_{\infty}}{(uxq^{\lambda_3},vyq^{\eta_2};q)_{\infty}} \sum_{k=0}^{\infty}
\frac{(wz)^k}{(uxq^{\beta_1},vyq^{\alpha_2},q;q)_k}\notag\\
&\hspace{0.5cm}\cdot\sum_{p=0}^{k}\left[\begin{matrix}
	k\\
	p
\end{matrix}\right]_q
(q^{\eta_2};q)_{k-p}(q^{\alpha_2-\eta_2};q)_p
q^{(k-p)(\alpha_2-\eta_2)},
\end{align*}
where $\displaystyle \left[\begin{matrix}
k\\
p
\end{matrix}\right]_q$ denotes the $q$-binomial coefficient \cite[p. 24]{Gasper-Rahman-2004}. Then the use of the $q$-binomial theorem \cite[p. 25]{Gasper-Rahman-2004} gives 
\begin{align}\label{Cor-qErdelyiIntegral-qFK-Proof-2}
\mathcal{K}(q;x,y,z;u,v,w)
&=\frac{(uxq^{\beta_1},vyq^{\alpha_2};q)_{\infty}}{(uxq^{\lambda_3},vyq^{\eta_2};q)_{\infty}} \sum_{k=0}^{\infty}
\frac{(q^{\alpha_2};q)_k(wz)^k}{(uxq^{\beta_1},vyq^{\alpha_2},q;q)_k}\notag\\
&=\frac{(uxq^{\beta_1},vyq^{\alpha_2};q)_{\infty}}{(uxq^{\lambda_3},vyq^{\eta_2};q)_{\infty}}
{}_3\phi_2\left[\begin{matrix}
	q^{\alpha_2},0,0 \\
	uxq^{\beta_1},vyq^{\alpha_2}
\end{matrix};q,wz\right].
\end{align}
Substituting \eqref{Cor-qErdelyiIntegral-qFK-Proof-2} into \eqref{Cor-qErdelyiIntegral-qFK-Proof-1} we obtain \eqref{Cor-qErdelyiIntegral-qFK-1}. 
\end{proof}

\section{Conclusion}

In this paper, we continue the recent work of Luo, Xu, and Raina \cite{Luo-Xu-Raina-2022} on Erd\'{e}lyi-type integrals for Saran's function $F_K$, providing a new proof of one of their main results (Theorem \ref{Th-ErdelyiIntegral-FK}). Building on this, we obtain an important generalization, which is applicable to the $L$-variable $F_K$ function \eqref{GeneralizedSaranFK} that arises in physics. Furthermore, some interesting $q$-analogues and discrete analogues are also discussed. 

By reviewing the literature on Erd\'{e}lyi-type integrals of various hypergeometric functions over the past century, this paper compiles the table in Appendix \ref{Appendix-1}, which reflects the overall landscape of research in this direction. As already clarified in Section \ref{Introduction}, whether out of computational interest or due to their usefulness, it is highly valuable to continue exploring Erd\'{e}lyi-type integrals for other  hypergeometric functions not yet included in this table.

\newpage
\section{Appendix}\label{Appendix-1}

\begin{center}
	\begin{tabular}{clll}
		\toprule
		\textbf{Type of function}  &  &  & \textbf{Reference} \\
		\midrule
		$1$ variable  & ${}_{2}F_{1}$ &  & Erd\'{e}lyi (1939) \cite{Erdelyi-1939a}, \cite{Erdelyi-1939b}\\
		  &  &  & Joshi-Vyas (2003) \cite{Joshi-Vyas-2003} \\
		&  ${}_{p}F_{q}$ &  & Joshi-Vyas (2003) \cite{Joshi-Vyas-2003}\\
		  &  &  & Luo-Raina (2017) \cite{Luo-Raina-2017}\\
		& $q$-analogues   &  & Gasper (2000) \cite{Gasper-2000}\\
		  &  &  & Joshi-Vyas (2006) \cite{Joshi-Vyas-2006}\\
		  &  &  & Vyas (2024) \cite{Vyas-2024}\\
		  & discrete analogues &  & Vyas-Bhatnagar-Fatawat-Suthar-Purohit (2022) \cite{Vyas-Bhatnagar-Fatawat-Suthar-Purohit-2022}\\
		  &  &  & Bhatnagar-Vyas (2022) \cite{Bhatnagar-Vyas-2022}\\
		  & others &  & Laine (1982) \cite{Laine-1982}\\
		  &  &  & Virchenko-Rumiantseva (2008) \cite{Virchenko-Rumiantseva-2008}\\
		  &  &  & Virchenko-Ovcharenko (2011) \cite{Virchenko-Ovcharenko-2011}\\
		\midrule 
		$n$ variables & Appell & $F_1$ & Manocha (1967) \cite{Manocha-1967}\\
		($n\geq 2$)  &  &       & Mittal (1977) \cite{Mittal-1977}\\
		  &  & $F_2$ & Koschmieder (1947) \cite{Koschmieder-1947}\\
		  &  &       & Manocha (1967) \cite{Manocha-1967} \\
		  &  &       & Mittal (1977) \cite{Mittal-1977}\\
		  &  & $F_3$ & Mittal (1977) \cite{Mittal-1977}\\
		  &  & $F_4$ & Manocha (1965) \cite{Manocha-1965} \\
		  &  &       & Mittal (1977) \cite{Mittal-1977}\\
		& Lauricella & $F_A$ & Manocha-Sharma (1969) \cite{Manocha-1969}\\
		  &  &  & Chandel (1971) \cite{Chandel-1971}\\
		  &  & $F_B$ & Chandel (1971) \cite{Chandel-1971}\\
		  &  & $F_C$ & Chandel (1971) \cite{Chandel-1971}\\
		  &  & $F_D$ & Koschmieder (1962) \cite{Koschmieder-1962}\\
		  &  &       & Manocha-Sharma (1969) \cite{Manocha-1969}\\
		  &  &       & Chandel (1971) \cite{Chandel-1971}\\
		  &  &       & Khudozhnikov (2003) \cite{Khudozhnikov-2003}\\
		& Saran & $F_M$ & Manocha-Sharma (1969) \cite{Manocha-1969}\\
		&  & $F_K$ & Luo-Raina (2021) \cite{Luo-Raina-2021} \\
		&  &  &Luo-Xu-Raina (2022) \cite{Luo-Xu-Raina-2022}\\
		& $q$-analogues &  & \emph{the present paper}\\
		& others &  & Volkodavov-Nikolaev (1993) \cite{Volkodavov-Nikolaev-1993}\\
		\bottomrule
	\end{tabular}
\end{center}

\section*{Author Contributions}

All the authors have contributed equally. All authors have read and approved the final manuscript.

\section*{Funding}

The research of the second author is supported by National Natural Science Foundation of China (Grant No. 12001095).

\section*{Data Availability}

This manuscript has no associated data.

\section*{Acknowledgements}

We are grateful to Peng-Cheng Hang for his careful reading of the initial manuscript, for correcting several typos, and for providing valuable suggestions.

\section*{Conflicts of interest}

The authors declare that there is no conflict of interest.

\end{document}